\documentclass{amsart}

\usepackage[latin1]{inputenc}
\usepackage{amsthm} 
\usepackage{amsmath}
\usepackage{amssymb}
\usepackage{mathrsfs}
\usepackage{enumitem}

\usepackage[T1]{fontenc}

\usepackage{tikz}
\usetikzlibrary{matrix,arrows,decorations.pathmorphing}

\usepackage[hidelinks]{hyperref}

\theoremstyle{plain} 
\newtheorem{theo}{Theorem}[section] 
\newtheorem{prop}[theo]{Proposition} 
\newtheorem{lem}[theo]{Lemma} 
\newtheorem{cor}[theo]{Corollary} 

\theoremstyle{definition} 
\newtheorem{conj}[theo]{Conjecture} 
\newtheorem{defi}[theo]{Definition}

\theoremstyle{remark}

\DeclareMathOperator{\im}{Im}
\DeclareMathOperator{\re}{Re}
\DeclareMathOperator{\Li}{Li}

\DeclareMathOperator{\Var}{Var}
\DeclareMathOperator{\ord}{ord}
\DeclareMathOperator{\Sym}{Sym}
\DeclareMathOperator{\tr}{tr}

\newcommand*\diff{\mathop{}\!\mathrm{d}}

\begin{document}
	
	\title[Discrepancies in the distribution of Gaussian primes]{Discrepancies in the distribution \\ of Gaussian primes}
	\author{Lucile Devin}
	\email{lucile.devin@univ-littoral.fr} 
\address{Univ. Littoral C\^ote d'Opale, UR 2597
	LMPA, Laboratoire de Math\'ematiques Pures et Appliqu\'ees Joseph Liouville,
	F-62100 Calais, France
}
\address{CNRS -- Université de Montréal CRM - CNRS}
	
	\date\today

	\begin{abstract}
		Motivated by questions of Fouvry and Rudnick on the distribution of Gaussian primes, we develop a very general setting in which one can study inequities in the distribution of analogues of primes through analytic properties of infinitely many $L$-functions.
		In particular, we give a heuristic argument for the following claim :
		for more than half of the prime numbers that can be written as a sum of two squares, the odd square is the square of a positive integer congruent to $1 \bmod 4$.
	\end{abstract}

	\maketitle
	
	\section{Introduction}

Let $p$ be a prime number, $p\equiv 1 \bmod 4$, 
	it can be written uniquely as $p = a^2 + 4b^2$ with $a,b >0$ integers.
	Here we distinguish the odd and the even square to have uniqueness, and we are interested in properties of the coordinates $a$ and $2b$.
	For example, an open problem in this domain is to prove that there are infinitely many prime numbers of the form $p= 1 + 4b^2$. 
	Allowing a larger range for the coordinates, Fouvry and Iwaniec \cite{FouvryIwaniec_1997} have shown that there are infinitely many prime numbers  $p= a^2 + 4b^2$ with $a$ itself a prime number, while Friedlander and Iwaniec \cite{FriedlanderIwaniec_1998} established the case where one of $a$ or $2b$ is a square. 
	More recent results restricting the values of $a$ and $b$ to thinner sets include \cite{HBL,LSX,Pratt,FriedlanderIwaniec_2018}.

In another direction, starting from the result of Hecke \cite{Hecke} ensuring that the angles defined by $\arctan(\tfrac{2b}{a})$ are equidistributed, it is of interest to study finer statistics of this distribution.
Kubilius \cite{Kubilius50,Kubilius51} and Ankeny \cite{Ankeny} initiated the study of the distribution of these angles in shrinking sectors, and further developments followed \cite{Koval,Maknis,Coleman,HarmanLewis,HuangLiuRudnick,CKLMSSWWY}. 
Recently Rudnick and Waxman \cite{RudnickWaxman} studied the function field analogue and obtained statistical results that hold for almost all arcs of shrinking length.

In this paper we discuss the following two questions:
\begin{enumerate}
	\item how often is $a<2b$ compared to $a > 2b$? 
	\item how often is $a \equiv 1 \bmod 4$ compared to $a \equiv 3 \bmod 4$? 
\end{enumerate}

The results of Hecke \cite{Hecke} on the prime number theorem for $L$-functions associated to Hecke characters provide an answer to these two questions : asymptotically half of the time.
Inspired by the letter of Chebyshev to Fuss \cite{ChebLetter}, and the rich literature that followed on Chebyshev's bias (see \cite{MartinScarfy,MartinBiblio} for a survey on the numerous contributions to the questions on ``prime number races''), we investigate the discrepancy in these equidistribution results.
We denote the differences of counting functions
$$D_1(x) = \lvert\lbrace p < x : p= a^2 + 4b^2, \lvert a\rvert > \lvert 2b \rvert\rbrace\rvert - \lvert\lbrace p < x : p= a^2 + 4b^2, \lvert a\rvert < \lvert 2b \rvert\rbrace\rvert,$$
$$D_2(x) = \lvert\lbrace p < x : p= a^2 + 4b^2, \lvert a\rvert \equiv 1\bmod{4} \rbrace\rvert - \lvert\lbrace p < x : p= a^2 + 4b^2, \lvert a\rvert \equiv 3\bmod{4} \rbrace\rvert.$$
Then, assuming the Generalized Riemann Hypothesis, the results of Hecke imply that, for $i=1$, $2$, $D_i(x) = O_{\epsilon}(x^{\frac12 + \epsilon})$.
As observed in general for functions related to prime counting (see e.g. \cite{Win}), the logarithmic scale is the correct scale to study the functions $D_i$.
We recall the definition of the upper and lower logarithmic densities of a set $\mathcal{P} \subset [1,\infty)$: 
	 \[\overline{\delta}(\mathcal{P}) = \limsup_{Y\rightarrow\infty}\frac{1}{Y}\int_{0}^{Y}\mathbf{1}_{\mathcal{P}}(e^{y})\diff y \ \text{  and }
	\ \underline{\delta}(\mathcal{P}) = \liminf_{Y\rightarrow\infty}\frac{1}{Y}\int_{0}^{Y}\mathbf{1}_{\mathcal{P}}(e^{y})\diff y,\]
	where $\mathbf{1}_{\mathcal{P}}$ is the characteristic function of the set $\mathcal{P}$.
	If these two densities are equal, we denote $\delta(\mathcal{P})$ their common value and call it the logarithmic density of the set~$\mathcal{P}$.

In this paper, we give a heuristic model for the distribution of the values of the functions $D_1$ and $D_2$.
We formulate two conjectures.
	\begin{conj}\label{Conj bias even odd}
		There is a bias towards negative values in the distribution of the values of the function $D_1$.
		That is to say that for more than half of the $x \in [2,\infty)$ in logarithmic scale, more than half of the primes below $x$ can be written as a sum of two squares with the even square larger than the odd square. 
		
		However we have $D_1(x)= \Omega_{\pm}(\frac{\sqrt{x}}{\log x})$.
		In particular, the function changes signs infinitely often.
	\end{conj}

	\begin{conj}\label{Conj bias A mod 4}
	There is a complete bias towards positive values in the distribution of the values of the function $D_2$.
	That is to say that, for almost all (in logarithmic scale) $x\in [2,\infty)$, more than half of the primes below $x$ can be written as a sum of two squares $p=a^2+4b^2$ with $\lvert a\rvert \equiv 1\bmod 4$. 

Precisely, we have $D_2(x) \geq 0$ for almost all $x$ in logarithmic scale,  and $D_2(x) = \Omega_{+}(\frac{\sqrt{x}}{\log x})$.
\end{conj}

\begin{figure}
\includegraphics[scale=0.75]{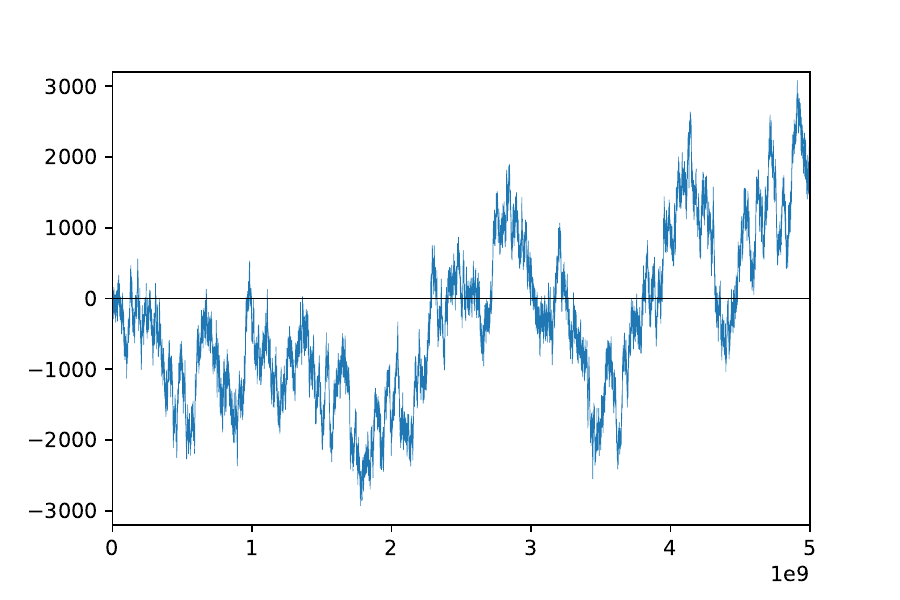}
\caption{$D_1(x)$ for $x \in [2,5\cdot 10^9]$.} 
\label{Fig race odd even}
\end{figure}

\begin{figure}
\includegraphics[scale=0.75]{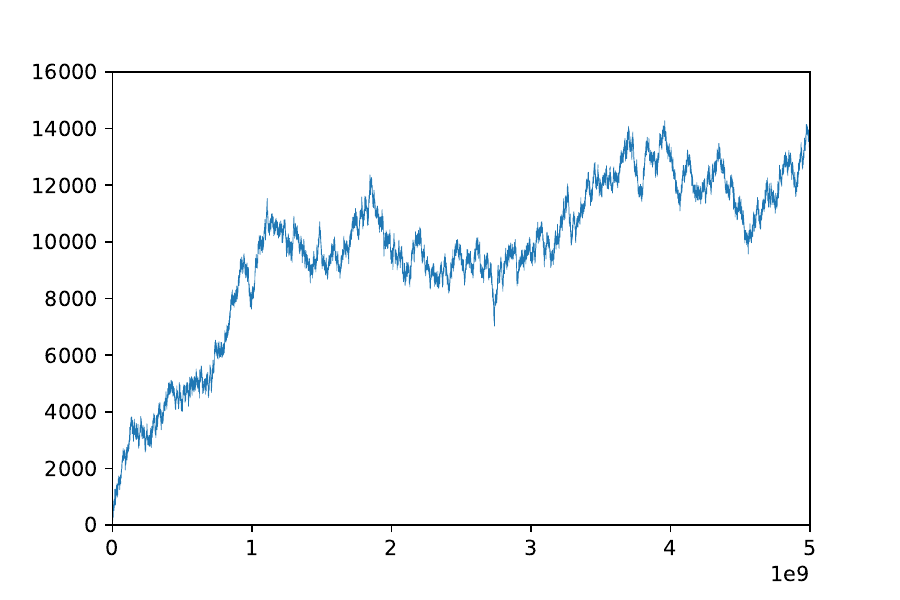}
\caption{$D_2(x)$ for $x \in [2,5\cdot 10^9]$.}
\label{Fig race A mod 4}
\end{figure}

The $\Omega$ notation in the Conjectures is used in the following sense. 
For two functions~$f,g : [1,\infty) \rightarrow \mathbf{R}$, with~$g$ positive, we write $f(x) = \Omega_{+}(g)$ (resp. $ \Omega_{-}(g)$) if we have
$\limsup_{x\rightarrow \infty} \frac{f(x)}{g(x)} >0$ (resp. $\liminf_{x\rightarrow \infty} \frac{f(x)}{g(x)}<0$).

Numerical data up to $5\cdot 10^9$ have been computed to support both conjectures and the graphs of the prime number races are presented in Figures~\ref{Fig race odd even} and~\ref{Fig race A mod 4}. Note in particular that no negative value of $D_2$ was found in the interval.

The paper is organized as follows. 
Section~\ref{sec General} is the technical heart of the paper where the setting and the precise statement of our main theoretical contribution (Theorem~\ref{Th_GeneDistLim}) are given. 
Here we emphasize the novel feature (compared to previous results e.g. \cite{RS,ANS,DevinChebyshev}) where we have managed to obtain distribution results involving infinitely many $L$-functions. 
In Section~\ref{sec cor Gaussian case}, we introduce the setting to study statistics about the distribution of Gaussian primes, and we state our results towards Conjecture~\ref{Conj bias even odd} and Conjecture~\ref{Conj bias A mod 4}: Theorem~\ref{Th race Gaussian primes} and Theorem~\ref{Th race Gaussian primes A mod 4}. 
These two results are proved in Section~\ref{Sec Hecke L functions} as consequences of Theorem~\ref{Th_GeneDistLim}. 
Building on the ideas underlying Theorem~\ref{Th race Gaussian primes} and Theorem~\ref{Th race Gaussian primes A mod 4}, we present in Section~\ref{Sec Heuristic} heuristics that led us to state Conjecture~\ref{Conj bias even odd} and Conjecture~\ref{Conj bias A mod 4}. 
Finally in Section~\ref{Section proof theo general}, we give the proofs of the theoretical contributions stated in Section~\ref{sec General}, notably Theorem~\ref{Th_GeneDistLim}.

\section*{Acknowledgements}
I am very grateful to Zeév Rudnick and Etienne Fouvry whose questions on very early versions of this project were the starting points for what became respectively Theorem~\ref{Th race Gaussian primes} and Conjecture~\ref{Conj bias A mod 4}.
This paper also benefited greatly from discussions with Chantal David, Daniel Fiorilli, Andrew Granville, Florent Jouve and Anders Södergren. I thank the referee for their careful reading and suggestions to improve the manuscript.
The author was supported successively by postdoctoral fellowships from the University of Ottawa, from the Centre de Recherches Mathématiques and from Chalmers University of Technology. The author is supported by the grant KAW 2019.0517 from the Knut and Alice Wallenberg Foundation.

The computations presented here were obtained using SageMath~\cite{Sage}, with the help of Valentin Priasso to deal with larger numbers and space, and using PARI/GP~\cite{PARI2}, with the help of Emmanuel Royer.

	\section{Chebyshev's bias using infinitely many $L$-functions}\label{sec General}

\subsection{Motivations and setting} 
	
	In \cite{MazurErrorTerm}, Mazur discussed prime number races for elliptic curves, 
	or more generally, for the Fourier coefficients of a modular form.
	For example, he studied graphs of functions
	\begin{align*}
		x\mapsto \lvert\lbrace p\leq x : a_{p}(E) >0 \rbrace\rvert -  \lvert\lbrace p\leq x : a_{p}(E) <0 \rbrace\rvert
	\end{align*}
	where $a_{p}(E) = p+1 - \lvert E(\mathbf{F}_{p})\rvert$, 
	for some elliptic curve $E$ defined over $\mathbf{Q}$, 
	and observed a bias towards negative values when the algebraic rank of the elliptic curve is large.
	In \cite{SarnakLetter}, Sarnak commented and explained Mazur's observations.
	His analysis of the prime number race involves the zeros of all the symmetric powers $L(\Sym^{n}E,s)$ of the Hasse--Weil $L$-function of $E/\mathbf{Q}$. 
	Sarnak notes that, in the case of an elliptic curve without complex multiplication, the corresponding distribution may have infinite variance and concludes that Mazur's race should be unbiased. 
	However, in the case of an elliptic curve with complex multiplication, even though an infinite number of $L$-functions is needed to understand Mazur's bias, Sarnak states that it would be possible to observe (and compute) an actual bias.
	In \cite{CFJ}, Cha, Fiorilli and Jouve develop these ideas in the context of elliptic curves over function fields. 
	The heuristics for Conjecture~\ref{Conj bias even odd}  (resp.~\ref{Conj bias A mod 4}) is the implementation of these ideas in the special case of the elliptic curve $y^2 = x^3 - x$ (resp. $y^2 = x^3 + x$) with complex multiplication by the ring of Gaussian integers~$\mathbf{Z}[i]$.
	
	In the spirit of \cite{DevinChebyshev}, the main result of this article is stated in greater generality and deals with the case of prime number races using infinitely many real analytic $L$-functions.	
	Before stating our theorem, let us present some definitions and notation.
	Our main result states the existence of a limiting logarithmic distribution for a suitable normalization of a counting function, and thus allow us to discuss the bias in the distribution of the values of this function.	
	\begin{defi}
		Let $F:[1,\infty)\rightarrow\mathbf{R}$ be a real function. 
		We say that $F$ admits a limiting logarithmic distribution
		if there exists a probability measure $\mu$ on the Borel sets in $\mathbf{R}$ such that
		for any bounded Lipschitz continuous function $g$, we have
		\begin{align*}
			\lim_{Y\rightarrow\infty}\frac{1}{Y}\int_{0}^{Y}g(F(e^{y}))\diff y = 
			\int_{\mathbf{R}}g(t)\diff\mu(t).
		\end{align*}
		If  $F$ admits a limiting logarithmic distribution $\mu$, we say that $\mu([0,\infty))$ is \emph{the bias} of $F$ towards non-negative values. 
	\end{defi}
	
	Note that our definition of the bias differs from previous literature where it is usually defined as the logarithmic density of the set of $x$ such that $F(x)>0$. However, in the setting of Chebyshev's bias, it is expected (see \cite{KR}, also \cite{RS} assuming the Linear Independence and \cite{MartinNg,DevinChebyshev} under weaker hypotheses) that in general the distribution $\mu$ is continuous at $0$ and then the two definition coincide.
	We thus consider it as a sufficiently good approximation.
	
	\subsection{Statement of the main result}
	
	In this paper we use the notion of \emph{analytic $L$-function} defined in \cite[Def. 1.1]{DevinChebyshev}. Let us just recall that a vast majority of the $L$-functions used in analytic number theory (and all the $L$-functions referred to in this article) are proven (or at least conjectured) to be analytic $L$-functions.
	We are interested in the prime number race associated to a sequence $\mathcal{S} =\lbrace L(f_m,\cdot) :  m\geq 0 \rbrace$ of real analytic $L$-functions of degrees $(d_m)_{m\geq 0}$, and analytic conductors $(\mathfrak{q}(f_m))_{m\geq 0}$. 
	To this sequence we associate real coefficients $\underline{c} = (c_{m})_{m\geq 0}$ such that
	the series 
	\begin{equation}
		\label{condition coeff}
		\sum_{m\geq 0}\lvert c_m \rvert d_m \log \mathfrak{q}(f_m) <\infty
	\end{equation} is convergent.
	Moreover, we assume that the Generalized Riemann Hypothesis is true for all the $L$-functions $ L(f_m,\cdot)$, $m\geq 0$ and their \emph{second moment}.
	 Recall from~\cite[Def.~1.1.(iii)]{DevinChebyshev} that the second moment of $L(f,\cdot)$ is defined\footnote{and that this possibly confusing name is due to \cite{Conrad}} as $L(f^{(2)},\cdot) := L(\Sym^2f,\cdot) L(\wedge^2f,\cdot)^{-1}$ when the $L$-function is seen as defined over rational primes; in particular if the generic local factors of $L(f,s)$ at rational primes are of the shape $\prod_{j=1}^{d}(1 - \alpha_{j,p}p^{-s})^{-1}$, those of $L(f^{(2)},s)$ are given by $\prod_{j=1}^{d}(1 - \alpha_{j,p}^2 p^{-s})^{-1}$.
	More precisely the assumption is that the non-trivial zeros of the functions $ L(f_m,\cdot)$, $ L(\Sym^2f_m,\cdot)$, and $ L(\wedge^2f_m,\cdot)$, $m\geq 0$ all have real part equal to $\tfrac12$.
	
	We note that, in this case, the difficulty compared to previous results is to deal with the fact that there can be an infinite number of zeros of the functions $L(f_m,\cdot)$, $m\geq 0$ with bounded imaginary part.
	Let us first fix the notation for the sets of zeros of the $L$-functions.
	For any $\rho\in \mathbf{C}$, we denote by $\ord(\rho,m) = \ord_{s= \rho}(L(f_m,s))$ the order of the zero at $s=\rho$ of the function $L(f_m,s)$. For $M\geq0$, we denote 
	\begin{align*}
		\ord_{\mathcal{S},\underline{c},M}(\rho) = \sum_{m\leq M} c_m \ord(\rho,m), \qquad
		\ord_{\mathcal{S},\underline{c}}(\rho) = \sum_{m\geq 0} c_m \ord(\rho,m).
	\end{align*}
	One has (see e.g. \cite[Prop.~5.7~(1)]{IK}),
	$$\ord(\rho,m) \ll \log(\mathfrak{q}(f_m)(\lvert\rho\rvert +3)^{d_m}), $$ 
	with an absolute implicit constant. Thus, the condition~\eqref{condition coeff} on the coefficients~$c_m$ ensures that the series defining $\ord_{\mathcal{S},\underline{c}}(\rho)$ is convergent for each~$\rho$.
	When we assume the Generalized Riemann Hypothesis, we omit the $\frac12$ in the notation and write for example, for $\gamma \in \mathbf{R}$,
	$\ord_{\mathcal{S},\underline{c}}(\gamma)$ instead of $\ord_{\mathcal{S},\underline{c}}(\tfrac12 +i\gamma)$.
	Then, for~$M\geq 0$ and $T >0$, we denote
	\begin{align*}
		\mathcal{Z}_{\mathcal{S},\underline{c}} = \lbrace \gamma>0 : \ord_{\mathcal{S},\underline{c}}(\gamma)\neq 0 \rbrace, \quad
		& \quad \mathcal{Z}_{\mathcal{S},\underline{c}}(T) = \mathcal{Z}_{\mathcal{S},\underline{c}}\cap(0,T] \\
		\mathcal{Z}_{\mathcal{S},\underline{c},M} = \lbrace \gamma>0 : \ord_{\mathcal{S},\underline{c},M}(\gamma)\neq 0\rbrace, \quad
		& \quad \mathcal{Z}_{\mathcal{S},\underline{c},M}(T) = \mathcal{Z}_{\mathcal{S},\underline{c},M}\cap(0,T],
	\end{align*}
	the sets of positive imaginary parts of zeros of the product of the $L$-functions in the family. We do not count with multiplicities in these sets. 
	
We now have all the tools to state our main theorem.
	
	\begin{theo}\label{Th_GeneDistLim}
		Let $\mathcal{S} =\lbrace L(f_m,\cdot) :  m\geq 0 \rbrace$ be a sequence of real analytic $L$-functions of degrees $(d_m)_{m\geq 0}$ and analytic conductors $(\mathfrak{q}(f_m))_{m\geq 0}$, 
		and let $\underline{c} = (c_{m})_{m\geq 0}$ be a sequence of real numbers such that
		the series $$\sum_{m\geq 0}\lvert c_m \rvert d_m \log \mathfrak{q}(f_m)$$ is convergent.
		Assume the Riemann Hypothesis is satisfied for all $L(f_m,\cdot)$, $m\geq 0$ and their second moment.
	Then the function
		\[E_{\mathcal{S},\underline{c}}(x):=\frac{\log x}{\sqrt{x}}\Bigg(
		\sum_{p\leq x}\sum_{m\geq 0}c_{m}\lambda_{f_m}(p) + \sum_{m\geq 0}c_{m}\ord_{s=1}(L(f_m,s)) \Li(x) \Bigg)\]
	 admits a limiting logarithmic distribution $\mu_{\mathcal{S},\underline{c}}$
		with average value
		\[\mathbb{E}(\mu_{\mathcal{S},\underline{c}}) = m_{\mathcal{S},\underline{c}} := \sum_{m\geq 0}c_{m}
		\left(\ord_{s=1}(L(f_m^{(2)},s)) 
		-2\ord_{s=\frac12}(L(f_m,s))\right)\]
		and variance
		\[\Var(\mu_{\mathcal{S},\underline{c}}) = 2\sum_{\gamma \in \mathcal{Z}_{\mathcal{S},\underline{c}}}  \frac{\lvert \ord_{\mathcal{S},\underline{c}}(\gamma)\rvert^{2}}{\frac14 +\gamma^{2}}.\] 
		Moreover, there exists a constant $a>0$ depending on $\mathcal{S}$ and $\underline{c}$, such that we have 
		\[ \mu_{\mathcal{S},\underline{c}}((-\infty,-R)\cup(R,\infty)) \ll \exp(-a\sqrt{R}).\]
	\end{theo}
	
	We prove this theorem in Section~\ref{Section proof theo general}.
	A way to understand this statement, is to think that the function $y\mapsto E_{\mathcal{S},\underline{c}}(e^y)$ takes its values with some probability law of mean value $m_{\mathcal{S},\underline{c}}$.
	In general (see \cite[Cor. 2.4]{DevinChebyshev}), we expect the probability law to be symmetric with respect to its mean value, so we think of the mean value as a good indicator of the behaviour of $y\mapsto E_{\mathcal{S},\underline{c}}(e^y)$, and the indication is more precise when the variance is relatively small.
	
	In the case of \cite{SarnakLetter}, $E/\mathbf{Q}$ is an elliptic curve without complex multiplication, and we consider the $L$-functions $L(f_1,\cdot) = L(E,\cdot)$ and  $L(f_m,\cdot) = L(\Sym^mE,\cdot)$ for $m\geq 2$. The degree of each $L$-function is $d_m = m+1$ and the analytic conductors satisfy $\log\mathfrak{q}(f_m) \ll m \log (m\mathfrak{q}(f_1)) $ (see \cite[(ii)]{SarnakLetter}). We observe that our condition on the convergence of $\sum_{f_m\in \mathcal{S}}\lvert c_m \rvert d_m \log \mathfrak{q}(f_m)$ corresponds to the condition stated in \cite[Cor. 2.9]{CFJ} --- precisely $c_m \ll m^{-3-\eta}$ for some $\eta >0$ --- to ensure the existence of the limiting distribution in the analogous question over function fields.
	
	\subsection{Signs changes}
	Most of the conditional results of \cite{RS,MartinNg,DevinChebyshev} giving more precisions on the properties of the limiting distribution $\mu_{\mathcal{S},\underline{c}}$ can be adapted to this case.
	Let us present here some results concerning the support of $\mu_{\mathcal{S},\underline{c}}$, as they can help to answer the question ``does the function have infinitely many sign changes?'', and provide Omega-results.
	Similarly to \cite[Th. 1.2]{RS}, we have a lower bound for the tails of the distribution, in case all the coefficients have the same sign.
	\begin{prop}\label{Prop lower bound tails}
		Let $\mathcal{S} =\lbrace L(f_m,\cdot) :  m\geq 0 \rbrace$ be a sequence of real analytic $L$-functions of degree $d_m$, and analytic conductor $\mathfrak{q}(f_m)$, 
		and let $\underline{c} = (c_{m})_{m\geq 0}$ be a sequence of \emph{non-negative} real numbers such that
		the series $$\sum_{m\geq 0} c_m  d_m \log \mathfrak{q}(f_m)$$ is convergent, and at least one $c_m \neq 0$.
		Assume the Riemann Hypothesis is satisfied for all $L(f_m,\cdot)$, $m\geq 0$ and their second moments.
		Then, there exist a constant $b>0$, depending on $\mathcal{S}$ and $\underline{c}$, such that we have
		$$ \min\big(\mu_{\mathcal{S},\underline{c}}((R,\infty)), \mu_{\mathcal{S},\underline{c}}((-\infty,-R))\big) \gg \exp(-\exp bR).$$
	\end{prop}
As we may want to use coefficients of different signs, we now state a result inspired by \cite[Th.~1.5(b and c)]{MartinNg} and using the notion of \emph{self-sufficient zero} that Martin and Ng introduced.
	\begin{defi}\label{Def_selfsuff}
We say that an ordinate $\gamma\in\mathcal{Z}_{\mathcal{S},\underline{c}}$ is self-sufficient if it is not in the $\mathbf{Q}$-span of $\mathcal{Z}_{\mathcal{S},\underline{c}}\smallsetminus\lbrace \gamma\rbrace$.
	\end{defi}
A priori, if there is no special reason for the imaginary parts of the zeros of some $L$-functions to be related, 
then we do not expect that there are any relation between them.
This general idea is called the General Simplicity Hypothesis, or the Linear Independence hypothesis (LI). It is used in particular in \cite{RS} to give more precisions on the limiting distribution in the original case of Chebyshev's bias as well as to compute an explicit value of this bias.
To state our next result, we do not need all the strength of this hypothesis. We show that
if there are enough self-sufficient zeros in $\mathcal{Z}_{\mathcal{S},\underline{c}}$ then the distribution~$\mu_{\mathcal{S},\underline{c}}$ is supported on all~$\mathbf{R}$.
	\begin{prop}\label{Prop inclusive}
		Let $\mathcal{S} =\lbrace L(f_m,\cdot) :  m\geq 0 \rbrace$ be a sequence of real analytic $L$-functions of degrees $(d_m)_{m\geq 0}$ and analytic conductors $(\mathfrak{q}(f_m))_{m\geq 0}$, 
		and let $\underline{c} = (c_{m})_{m\geq 0}$ be a sequence of real numbers such that
		the series $$\sum_{m\geq 0} \lvert c_m\rvert  d_m \log \mathfrak{q}(f_m)$$ is convergent.
		Assume the Riemann Hypothesis is satisfied for all $L(f_m,\cdot)$, $m\geq 0$ and their second moment.
		Let $\mathcal{Z}_{\mathcal{S},\underline{c}} = \lbrace \gamma>0 : \ord_{\mathcal{S},\underline{c}}(\gamma)\neq 0 \rbrace$
		and $\mathcal{Z}_{\mathcal{S},\underline{c}}^{\mathrm{LI}}$ the set of self-sufficient elements in~$\mathcal{Z}_{\mathcal{S},\underline{c}}$.
		Assume that $$\sum_{\gamma \in \mathcal{Z}_{\mathcal{S},\underline{c}}^{\mathrm{LI}}} \frac{\lvert \ord_{\mathcal{S},\underline{c} }(\gamma)\rvert}{\gamma}$$ diverges. 
		Then
		$ \mathrm{supp}(\mu_{\mathcal{S},\underline{c}}) = \mathbf{R}$.
	\end{prop}

	In particular, under such conditions, we deduce that there cannot be a complete bias, that is to say that the function $E_{\mathcal{S},\underline{c}}$ changes sign infinitely many times, and we obtain Omega-results.
	\begin{cor}\label{Cor Omega result}
		Under the assumptions of Proposition~\ref{Prop lower bound tails} or of Proposition~\ref{Prop inclusive},
		we have
		\begin{align*}
			\sum_{p\leq x}\sum_{m\geq 0}c_{m}\lambda_{f_m}(p) + \sum_{m\geq 0}c_{m}\ord_{s=1}(L(f_m,s)) \Li(x)  = \Omega_{\pm}\big(\tfrac{\sqrt{x}}{\log x}\big).
		\end{align*}
		In particular, 	the function 
		$$x \mapsto \sum_{p\leq x}\sum_{m\geq 0}c_{m}\lambda_{f_m}(p) + \sum_{m\geq 0}c_{m}\ord_{s=1}(L(f_m,s)) \Li(x) $$ 
		has infinitely many sign changes.
	\end{cor}
	
	Note that, as suggested in \cite{BestTrud}, it may be possible to estimate the implicit constants in the Omega bounds given enough explicit information on the zeros of the involved $L$-functions. Moreover 
	Finally, we observe that in the case of a non-real\footnote{The case of a real zero of maximal real part would give only one direction of Omega-result as in \cite[Lem. 2.1]{FiorilliMartin}.} counter-example to the Generalized Riemann Hypothesis, oscillations results are more easily obtained thanks to Landau's Theorem.
	
	\begin{prop}\label{Prop oscillation without GRH}
		Let $\mathcal{S} =\lbrace L(f_m,\cdot) :  m\geq 0 \rbrace$ be a sequence of real analytic $L$-functions of degree $d_m$ and analytic conductor $\mathfrak{q}(f_m)$, 
		and let $\underline{c} = (c_{m})_{m\geq 0}$ be a sequence of real numbers such that
		the series $$\sum_{m\geq 0} \lvert c_m\rvert  d_m \log \mathfrak{q}(f_m)$$ is convergent.
		Let $\Theta = \sup\lbrace \re(\rho) : \ord_{\mathcal{S},\underline{c}}(\rho) \neq 0\rbrace $. Assume that $\Theta>\frac12$ and that for each $m\geq 0$, one has $L(f_m,\Theta) \neq 0$.
Then for any $\epsilon >0$,	we have
\begin{align*}
	\sum_{p\leq x}\sum_{m\geq 0}c_{m}\lambda_{f_m}(p) + \sum_{m\geq 0}c_{m}\ord_{s=1}(L(f_m,s)) \Li(x)  = \Omega_{\pm}(x^{\Theta -\epsilon}).
\end{align*}
	\end{prop}

	\section{Distribution of the angles of Gaussian primes}\label{sec cor Gaussian case}
	
	The representation of a prime number as a sum of two squares is explained in the ring of Gaussian integers $\mathbf{Z}[i]$ where the prime numbers $p\equiv 1 \bmod 4$ split as $p =(a + i2b)(a - i2b)$, while the prime numbers $p \equiv 3 \bmod 4$ are inert. 
	The aim of Conjectures~\ref{Conj bias even odd} and~\ref{Conj bias A mod 4} is to understand fine statistics on the distribution of Gaussian primes in the plane.
	To do so, we study the distribution of the angles of the Gaussian primes which are defined as follows.
	For any $\mathfrak p \neq (1+i)$ prime ideal in $\mathbf{Z}[i]$, there exists a unique Gaussian integer $a + 2ib \equiv 1 \bmod (2+2i)$ generating $\mathfrak p$ (starting from $a,b \geq 0$ then either $a + 2ib \equiv 1 \bmod (2+2i)$ or $-a - 2ib \equiv 1 \bmod (2+2i)$).
	 Then we say that the angle of $\mathfrak{p}$ is the argument of this uniquely determined generator,
	\begin{equation}
	\label{Equation def angle}
	\theta_{\mathfrak{p}} = \arg(a + 2ib)  \in (-\pi,\pi] \text{ where } \mathfrak{p} = (a+2ib), \quad a+2ib \equiv 1 \bmod (2+2i).
	\end{equation}
	Note that, if $\mathfrak{p}$ is generated by $a + 2ib \equiv 1 \bmod (2+2i)$, then one also has $a - 2ib \equiv 1 \bmod (2+2i)$ and this latter Gaussian integer generates $\overline{\mathfrak{p}}$.
	So for a rational prime $p \equiv 1 \bmod 4$, one can define its Gaussian angle up to its sign, $\theta_p = \pm \theta_{\mathfrak{p}} \in [0,\pi]$, where $\mathfrak{p}\mid p$.
	We observe also that this choice is natural : the number of $\mathbf{F}_p$ points of the elliptic curve with affine model $y^2 = x^3 -x$ is exactly given by $p+1 - 2\sqrt{p}\cos(\theta_{p})$.
	Note finally that with this definition, the natural angle associated to an inert prime $p \equiv 3 \bmod 4$ is $\theta_p = \pi$. 
	
	Hecke proved in \cite{Hecke} that the angles of Gaussian primes equidistribute on the circle.
	In particular, for $0\leq \alpha < \beta \leq \pi$, one has
	\begin{equation*}
	\lim_{X\rightarrow\infty} \frac{\lvert p \leq X : p\equiv 1 \bmod 4, \alpha < \theta_p < \beta \rvert}{\lvert p \leq X : p\equiv 1 \bmod 4 \rvert} = \frac{\beta - \alpha}{\pi}.
	\end{equation*}
	The limit depends only on the length of the interval. 
In the context of Conjecture~\ref{Conj bias even odd}, we are interested in the error term in this result, and in particular to show how it depends on the interval $[\alpha,\beta]$.
Our main heuristic for this conjecture is the following result.	
		
	\begin{theo}\label{Th race Gaussian primes}
		Assume the Generalized Riemann Hypothesis.
		Let $\phi$ be a $2\pi$-periodic even function on $\mathbf{R}$, 
		with Fourier coefficients $\hat{\phi}(m) :=  \frac{1}{2\pi}\int_{-\pi}^{\pi} \phi(t)e^{2i\pi mt} \diff t$ for $m \in \mathbf{Z}$,
		satisfying $\hat{\phi}(m) \ll \lvert m\rvert^{-1 -\epsilon}$ for some $\epsilon >0$.
		Then the function 
		\begin{align*}
		E_{\phi}(x) := \frac{\log x}{\sqrt {x}}\Bigg(\sum_{\substack{p\leq x \\ p \equiv 1 \bmod{4}}} \phi(\theta_p) - \Li(x)\int_{0}^{\pi}\phi(t) \frac{\diff t}{2\pi} \Bigg)
		\end{align*}  
		admits a limiting logarithmic distribution $\mu_{\phi}$ as $x\rightarrow\infty$.
		
		Moreover, let $\xi$ be the Hecke character defined in~\eqref{def xi} and
		assume that the vanishing at $s= \frac12$ of the Hecke $L$-functions $L(s,\xi^m)$ of the powers $\xi^m$ is exactly given by the sign of its functional equation.
		Then
		the average value of $\mu_{\phi}$ is equal to 
			$$ \frac{-\hat{\phi}(0)}{2} - \frac{\phi(0) + \phi(\pi)}{4} +\frac{1}{2\pi} \int_{-\pi}^{\pi} \phi(t)\frac{\cos(t)}{\cos(2t)} \diff t, $$  
		where the integral is understood as the Cauchy principal value.
	\end{theo}
	Theorem~\ref{Th race Gaussian primes} is just an $\epsilon$ away from Conjecture~\ref{Conj bias even odd}. Indeed, we would like to choose $\phi$ to be a step function but the $m$-th Fourier coefficient of such a function is of size $\lvert m\rvert^{-1}$. This will be discussed further in Section~\ref{Sec Heuristic}.

However, as the function $\phi$ can be chosen very general, it actually gives fine information on the distribution of the angles $\theta_p$. 	
Figure~\ref{fig hist angles} shows the difference between a histogram of the distribution of the angles $\theta_{p}$ and equidistribution. We observe some irregularities in the otherwise relatively well equidistributed behaviour happening at $\frac{\pi}{4}$ and $\frac{3\pi}{4}$ corresponding to the poles of $\frac{\cos(t)}{\cos(2t)}$. 
This irregularity might only be due to our ``unfolding'' of the angle $\theta_p$ (observe that the angle $2\theta_p$ or even $4\theta_p$ is more usually studied in the literature \cite{Kubilius50,Coleman,RudnickWaxman} and does not exhibit such a bias). We note that our choice of $a + 2ib \equiv 1 \bmod (2+2i)$ imply that the only possibilities for $\lvert a\rvert$ and $\lvert 2b\rvert$ to be very close to each other are $a= \pm2b+1$  while $a = \pm 2b -1$ are excluded. This could explain the behaviour of the distribution around  $\frac{\pi}{4}$ and $\frac{3\pi}{4}$.

\begin{figure}
	\includegraphics[scale=0.65]{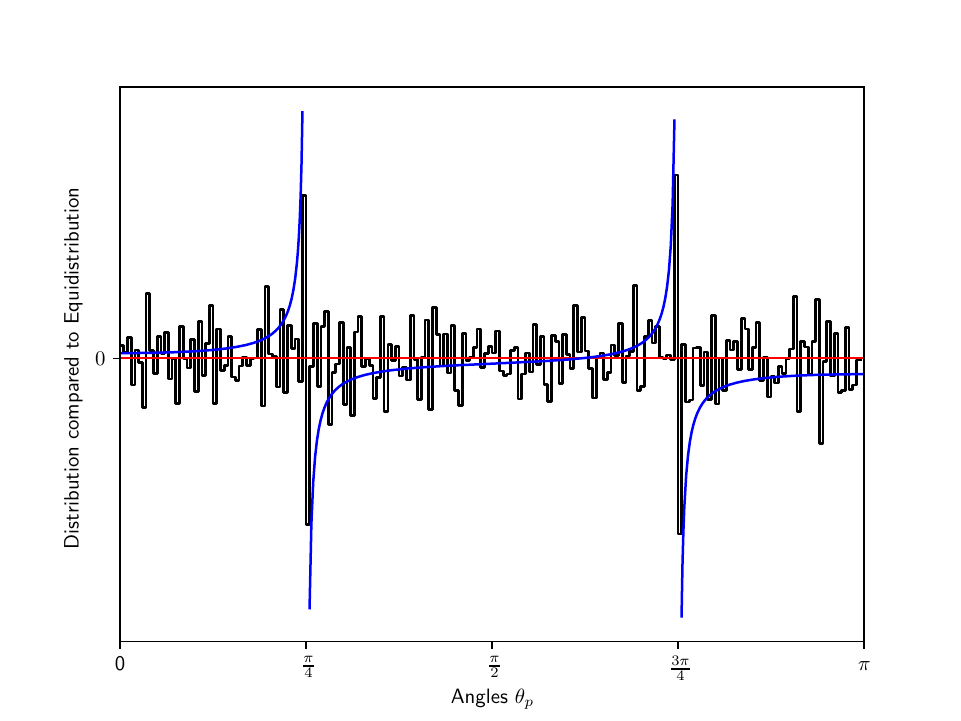}
	\caption{Relative distribution of the angles $\theta_{p}$ for $p\leq 10^8$ : we count the number of angles $\theta_p$ in $200$ subintervals of $[0,\pi]$ and withdraw the mean value; in red equidistribution;  in blue the ``secondary term'' $\frac{\cos x}{\cos 2x} - \frac12$.}
	\label{fig hist angles}
\end{figure}
		
		In the case  $\phi$ is $\pi$-periodic with mean value $0$, we find that the mean value of $\mu_{\phi}$ is equal to $\frac{- \phi(\pi)}{2}$, this indicates a bias in the distribution of the angles $\theta_p$ due to the fact that our sum defining $E_{\phi}$ does not include the inert primes $p\equiv 3 \bmod 4$.

		Note finally that this is a case where the function field analogue differs from the original question. Indeed, \cite[Th. 1.8]{Perret-Gentil} shows that there is a bias in the distribution of the analogue of angles of Gaussian primes in function fields that is in the direction of sectors parametrized by non-squares. Such a phenomenon does not seem to appear here.\\

In the context of Conjecture~\ref{Conj bias A mod 4}, let us first 
note that for $p = a^2 + 4b^2$ with $a+2ib \equiv 1 \bmod (2+2i)$ then $p \equiv 1 \bmod 8 \Leftrightarrow a \equiv 1 \bmod 4$ and $p \equiv 5 \bmod 8 \Leftrightarrow a \equiv -1 \bmod 4$, 
so we have 
\begin{align*}
\sum_{p = a^2 + 4b^2 \leq x} \Big(\mathbf{1}_{\lvert a\rvert \equiv 1 \bmod 4}(p) -  \mathbf{1}_{\lvert a\rvert \equiv -1 \bmod 4}(p) \Big)&= \sum_{\substack{p\leq x \\ p \equiv 1 \bmod{8} }} \phi(\theta_p) - \sum_{\substack{p\leq x \\ p \equiv 5 \bmod{8} }} \phi(\theta_p)
\\ &= \sum_{\substack{p\leq x \\ p \equiv 1 \bmod{4} }} \phi(\tilde{\theta}_p),	
\end{align*}
where $\phi = \mathbf{1}_{(0,\frac{\pi}{2})}- \mathbf{1}_{(\frac{\pi}{2},\pi)}$ (so that for $a\equiv 1 \bmod 4$ one counts $+1$ when $a$ is positive, and $-1$ when it is negative; and the other way round for $a\equiv -1 \bmod 4$) and  we define $\tilde{\theta}_p = \begin{cases}
	\theta_{p} & \text{ if } p \equiv 1 \bmod{8} \\
	\pi - 	\theta_{p} & \text{ if } p \equiv 5 \bmod{8} 
\end{cases}$.     
Similarly to Theorem~\ref{Th race Gaussian primes}, we have an analogous result in the case $\phi$ is just slightly smoother than a step function.  

\begin{theo}\label{Th race Gaussian primes A mod 4}
	Assume the Generalized Riemann Hypothesis.
	Let $\phi$ be a $2\pi$-periodic even function on $\mathbf{R}$
	with Fourier coefficients $\hat{\phi}(m) :=  \frac{1}{2\pi}\int_{-\pi}^{\pi} \phi(t)e^{2i\pi mt} \diff t$ for $m \in \mathbf{Z}$,
	satisfying $\hat{\phi}(m) \ll \lvert m\rvert^{-1 -\epsilon}$ for some $\epsilon >0$.
	Then the function 
	\begin{align*}
		F_{\phi}(x) := \frac{\log x}{\sqrt {x}}\Bigg(\sum_{\substack{p\leq x \\ p \equiv 1 \bmod{8} }} \phi(\theta_p) - \sum_{\substack{p\leq x \\ p \equiv 5 \bmod{8} }} \phi(\theta_p) \Bigg)
	\end{align*}  
	admits a limiting logarithmic distribution $\nu_{\phi}$ as $x\rightarrow\infty$.
	
	Moreover, let $\psi_m$ be the Hecke characters defined in~\eqref{def psi_m} for $m\geq 0$, and
	assume that the vanishing at $s= \frac12$ of the Hecke $L$-functions $L(s,\psi_m)$ is exactly given by the sign of its functional equation.
	Then the average value of $\nu_{\phi}$ is equal to 
	$$\frac{-\hat{\phi}(0)}{2} -\frac{\phi(0) + \phi(\pi)}{4} +\frac{1}{2\pi} \int_{-\pi}^{\pi} \phi(t)\frac{1}{2\cos(t)} \diff t,$$
			where the integral is understood as the Cauchy principal value.
\end{theo}

 Taking $\phi$ such that $\phi(\pi-\theta)= - \phi(\theta)$ in Theorem~\ref{Th race Gaussian primes A mod 4} gives fine information on the distribution of the angles $\tilde{\theta}_p$. 	
Figure~\ref{fig hist angles Amod4} shows the difference between a histogram of the distribution of the angles $\tilde{\theta}_{p}$ and equidistribution. We observe some irregularity in the otherwise relatively well equidistributed behaviour happening at~$\frac{\pi}{2}$ corresponding to the pole of $\frac{1}{\cos(t)}$. 
Note that along the prime numbers $p= 1+4b^2$ we have $\tilde{\theta}_{p} \rightarrow \frac{\pi}{2}^-$. The jump in the distribution goes well with the idea that there are infinitely many such prime numbers.
\begin{figure}
	\includegraphics[scale=0.65]{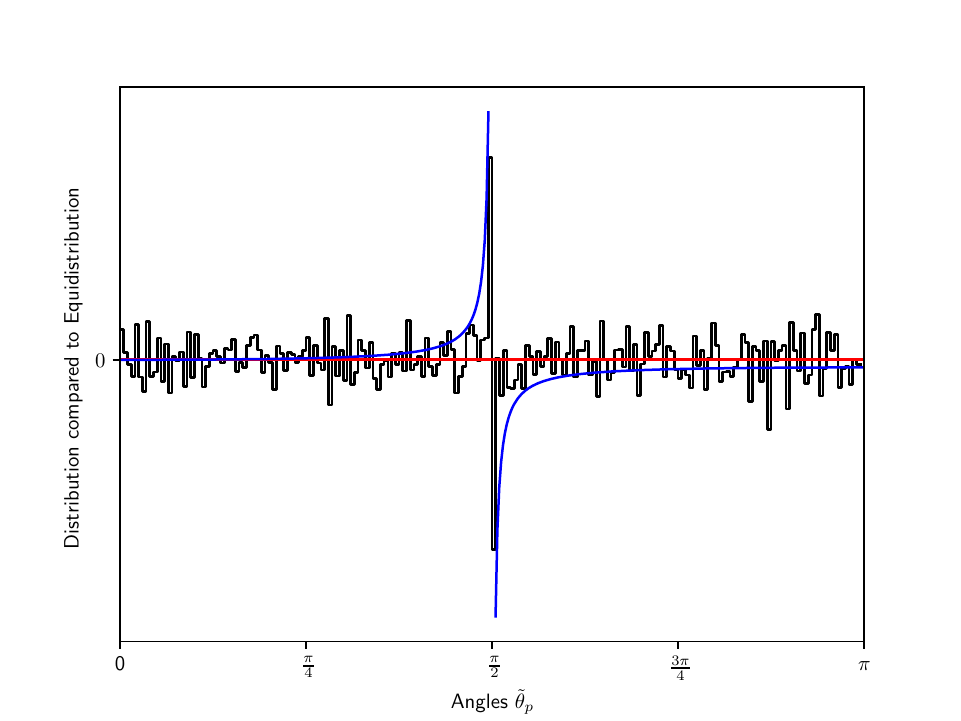}
	\caption{Relative distribution of the angles $\tilde{\theta}_{p}$ for $p\leq 10^8$ : we count the number of angles $\tilde{\theta}_p$ in $200$ subintervals of $[0,\pi]$ and withdraw the mean value; in red equidistribution;  in blue the ``secondary term'' $\frac{1}{\cos x} - \frac12$.}
	\label{fig hist angles Amod4}
\end{figure}

Theorem~\ref{Th race Gaussian primes} and~\ref{Th race Gaussian primes A mod 4} both follow from the decomposition of the function $\phi$ in Fourier series and an application of Theorem~\ref{Th_GeneDistLim} to a sequence of $L$-functions associated to the powers of a given Hecke character, with coefficients the Fourier coefficients of $\phi$.
We give the details on the precise Hecke characters in Section~\ref{Sec Hecke L functions}.
We also comment there on the assumption that allows us to give a formula for the mean value of the limiting distribution.

	\section{$L$-functions of Hecke characters on $\mathbf{Z}[i]$}\label{Sec Hecke L functions}
	
	\subsection{Properties of some Hecke characters and their $L$-function}
	
		Let us first review some properties of Hecke characters that will be useful in this paper. Our references for this section are mostly \cite{Rohrlich}, \cite[Chap.~VII.6]{Neukirch} and \cite[Chap.~3.8, Chap.~5.10]{IK}.
	Let $\eta$ be a unitary primitive Hecke character on $\mathbf{Z}[i]$ of conductor $\mathfrak{f}$ and frequency $\ell$, we define the Hecke $L$-function associated to $\eta$ as the following Dirichlet series or Euler product on the half-plane $\re(s)>1$:
	\begin{equation}
		L(s,\eta) = \sum_{\mathfrak n}\eta(\mathfrak n) N(\mathfrak n)^{-s} = \prod_{\mathfrak p}\left(
		1 - \eta(\mathfrak p)N(\mathfrak{p})^{-s} \right)^{-1}
	\end{equation}
	where $\mathfrak{n}$ (resp. $\mathfrak{p}$) runs over all non-zero (resp. prime) ideals of $\mathbf{Z}[i]$.
	Hecke proved that $L(s,\eta)$ extends meromorphically to the whole plane $\mathbf{C}$ with at most a simple pole at $s=1$ only when $\eta$ is the trivial character.
	Moreover, the completed $L$-function 
	$$\Lambda(s,\eta) := (4 N(\mathfrak{f}))^{s/2}\pi^{-s - (\lvert\ell\rvert + 1)/2}\Gamma\Big( \tfrac{s + \frac{\lvert \ell \rvert}{2}}{2}\Big)\Gamma\Big( \tfrac{s + 1+\frac{\lvert \ell \rvert}{2}}{2}\Big) L(s,\eta),$$ 
	admits a functional equation
	$\Lambda(\eta, s) = W(\eta)\Lambda(\overline{\eta},1-s)$,
	where $W(\eta)$ is a complex number of norm $1$ (the sign of the functional equation) that can be given explicitly via a Gauss sum \cite[(3.85)]{IK}.
	In the case $\eta(\overline{\mathfrak{p}}) = \overline{\eta}(\mathfrak{p})$ for all~$\mathfrak p$, one has $\Lambda(\eta, s) = \Lambda(\overline{\eta}, s)$ and $W(\eta) = \pm 1$.

\subsubsection{A character for angles of Gaussian primes}
	
		In equation~\eqref{Equation def angle}, the definition of the angle of Gaussian primes is given so that the function $\mathfrak{p} \rightarrow e^{i\theta_{\mathfrak{p}}}$ comes from a Hecke character.
	Precisely, let $\mathfrak{m} = (2+2i)$, there are exactly $4$ invertible elements in $\mathbf{Z}[i]/\mathfrak{m}$ corresponding to the $4$ units of $\mathbf{Z}[i]$ : $\pm 1, \pm i$.
	Thus, every ideal $\mathfrak{a} \subset \mathbf{Z}[i]$ co-prime to $\mathfrak{m}$ has a unique generator $\alpha \equiv 1 \bmod \mathfrak{m}$.
	Let $\xi$ be the Hecke character on the multiplicative groups of fractional ideals of $\mathbf{Z}[i]$ modulo $\mathfrak{m}$ defined by
	\begin{equation}\label{def xi}
		\xi((\alpha)) = \begin{cases}
			\frac{\alpha}{\lvert \alpha \rvert} &\text{ if } \alpha \equiv 1 \bmod \mathfrak{m} \\
			0 &\text{ if } (\alpha, \mathfrak{m}) \neq 1.
		\end{cases}
	\end{equation}
	So that for any prime ideal $\mathfrak{p}$ co-prime to $2$, one has $\xi(\mathfrak{p}) = e^{i\theta_{\mathfrak{p}}}$.
	Then $\xi$ is a unitary Hecke character of frequency $1$, and conductor $\mathfrak{m}$, its finite part is the Dirichlet character $\chi: u \mapsto u^{-1}$ for $u \in\lbrace \pm 1, \pm i\rbrace$ representing the four invertible congruence classes modulo $\mathfrak{m}$. 
	Coming back to our decomposition, for any unramified splitting rational prime $(p) = \mathfrak{p}\overline{\mathfrak{p}}$, one has, for any $m \in \mathbf{N}$  
	\begin{align*}
		2\cos(m\theta_{p}) = \xi^{m}(\mathfrak{p}) + \xi^{m}(\overline{\mathfrak{p}}).
	\end{align*}
	Moreover, for any integer $m \geq 1$, the character $\xi^{m}$ has frequency $m$ and finite part $\chi(u) = u^{-m}$ for $u \in\lbrace \pm 1, \pm i\rbrace$ (again seen as elements in  $\mathbf{Z}[i]/\mathfrak{m}$), it is primitive  of conductor respectively $(2+2i)$ if $m$ is odd, $(2)$ if $m \equiv 2 \bmod 4$ and non-primitive of conductor $(1)$ if $m\equiv 0 \bmod 4$. Let us denote $\xi_m$ the primitive character associated to $\xi^m$, we then have
	\begin{align*}
		\xi^m = \begin{cases}
			\xi_m &\text{ if } m \not\equiv 0 \bmod 4, \\
			\xi_m \chi_0 &\text{ if } m \equiv 0 \bmod 4
		\end{cases}
	\end{align*} 
	where $\chi_0$ is the principal character modulo $(1+i)$.
	In particular, 
	\begin{align*}
		\sum_{N\mathfrak{p} \leq x} \xi^m(\mathfrak{p}) = \sum_{N\mathfrak{p} \leq x} \xi_m(\mathfrak{p}) + O(1),
	\end{align*}
	so we will only loose a small error term in considering the primitive character $\xi_m$ instead of $\xi^m$.
	
In order to prove Theorem~\ref{Th race Gaussian primes}, we will apply Theorem~\ref{Th_GeneDistLim} to the family of $L$-functions attached to the $\xi_m$'s. Let us first review the properties of those $L$-functions.	

\begin{lem}\label{Lem L-func Xi}
	The $L$-function associated to $\xi_0$ is 
	$L(s,\xi_{0}) = \zeta_{\mathbf{Q}(i)}(s)$ and
	for $m\geq 1$, the $L$-functions
	\begin{multline*}
		L(s,\xi_{m}) = (1 - \xi_m(1+i)2^{-s})^{-1} \prod_{p \equiv 1 \bmod 4} (1 - 2\cos(m\theta_p) p^{-s} + p^{-2s})^{-1} \\ \times \prod_{p \equiv 3 \bmod 4} (1 - (-1)^m p^{-2s})^{-1} 
	\end{multline*}
	are analytic $L$-functions in the sense of \cite[Def. 1.1]{DevinChebyshev}.
	
	Moreover  each of these $L$-functions has degree $d_m=2$ and analytic conductor given by $$\mathfrak{q}(\xi_m) = \begin{cases}
		4 (\tfrac{m}{2} +3)(\tfrac{m}{2} + 4) & \text{ if } m \equiv 0 \bmod 4,\\
		16 (\tfrac{m}{2} +3)(\tfrac{m}{2} + 4) & \text{ if } m \equiv 2 \bmod 4,\\
		32 (\tfrac{m}{2} +3)(\tfrac{m}{2} + 4) & \text{ if } m \equiv 1 \bmod 2.\\
	\end{cases}$$
\end{lem}

\begin{proof}
See e.g. \cite[\S 5.10]{IK}.
The analytic conductor is given as in \cite[(5.7)]{IK} by the product of the absolute value of the discriminant of the base field multiplied by the norm of the conductor of the character $\xi_m$ 	and by the contribution of the $\Gamma$-factors as given in the end of \cite[\S 5.10]{IK}. 
\end{proof}

We then need to find out the order of vanishing $\ord_{s= \frac12}(L(s,\xi_{m}))$ of these functions at their central point. Our strategy consists in computing the sign of their functional equation which we denote $W(\xi_m)$ and assume that this knowledge is enough, see discussion in Section~\ref{subsection hypothese}.

\begin{lem}\label{Lem Sign Xi}
	The sign of the functional equation for $L(s,\xi_{m})$ depends on the congruence class of $m$ modulo $8$. 
	Precisely, one has
	\begin{align*}
		W(\xi_m) = \begin{cases}
			1 &\text{ if } m \equiv 0, 1, 2, 3, 4, 6 \bmod 8, \\
			-1 &\text{ if } m \equiv 5 , 7 \bmod 8.  
		\end{cases}
	\end{align*}
\end{lem}

\begin{proof}
The formula \cite[(3.85)]{IK} for the sign of the functional equation gives
\begin{align*}
	W(\xi_m) = i^{-m} N(\mathfrak{f}_m)^{-\frac12} \xi_{m,\infty}(\gamma_m) \sum_{x \in \mathbf{Z}[i]/\mathfrak{f}_m}\xi_{m,\mathrm{fin}}(x) e^{2i\pi \tr\big(\tfrac{x}{\gamma_m}\big)}
\end{align*}
where $\mathfrak{f}_m$ is the conductor of $\xi_m$ --- that is $(2+2i)$ is $m$ is odd, $(2)$ if $m \equiv 2 \bmod 4$ and $(1)$ if $m \equiv 0 \bmod 4$ ---, $\xi_{m,\infty}$ and $\xi_{m,\mathrm{fin}}$ are respectively the infinite and finite part of~$\xi_m$, and $\gamma_m \in \mathbf{Z}[i]$ is such that $(\gamma_m) = 2\mathfrak{f}_m$.
In particular for odd $m$,
\begin{align*}
	W(\xi_m) &= i^{-m} 8^{-\frac12} e^{i\pi \frac{m}{4}} \sum_{x \in \lbrace \pm 1, \pm i\rbrace} x^{-m} e^{2i\pi \tr\big(\tfrac{x}{4+4i}\big)} \\
	&= \frac{e^{-i\pi \frac{m}{4}}}{2\sqrt{2}} (2i - 2i^{m+1}) = \begin{cases}
		e^{i\pi \frac{1-m}{4}} &\text{ if } m\equiv 1 \bmod 4, \\
		e^{i\pi \frac{3-m}{4}} &\text{ if } m\equiv 3 \bmod 4.
	\end{cases}
\end{align*}
For $m \equiv 2 \bmod 4$, we have
\begin{align*}
	W(\xi_m) = i^{-m} 4^{-\frac12} \sum_{x \in \lbrace 1,i\rbrace}x^{-m} e^{2i\pi \tr(\tfrac{x}{4})} 
	=1,
\end{align*}
and for $m \equiv 0 \bmod 4$,
\begin{align*}
	W(\xi_m) = i^{-m} 1^{-\frac12} e^{2i\pi \tr(\tfrac{1}{2})} = 1.
\end{align*}
This yields the announced result.
	\end{proof}
	
	Finally, we express the order of vanishing $\ord_{s=1}(L(s,\xi_m^{(2)}))$ of the second moment $L$-functions.
	
	\begin{lem}\label{Lem 2nd at 1 Xi}
		For $m \geq 0$ we have
		$$\ord_{s=1}(L(s,\xi_m^{(2)})) = \begin{cases}
			-2 &\text{ for } m=0, \\
			-1 &\text{ for } m >0, \text{ even,}\\
			1 &\text{ for } m \text{ odd.}
		\end{cases}$$
		
	\end{lem}
	
	\begin{proof}
			In the case $m=0$, we have $L(s,\xi_{0}) = \zeta_{\mathbf{Q}(i)}(s) = \zeta(s) L(s,\chi_4)$ where $\chi_4$ is the non-trivial Dirichlet character modulo $4$. Up to the factor at~$2$, we have $L(s,\chi_4^{(2)}) = \zeta(s)$, and we get that $\ord_{s=1}(L(s,\xi_0^{(2)}))  = -2$. 
			
		For $m\geq 1$, we have that
		\begin{align*}
			L(s,\xi^{m  (2)}) &= \prod_{p \equiv 1 \bmod 4} (1 - e^{i2m\theta_p} p^{-s})^{-1}(1 - e^{-i2m\theta_p} p^{-s})^{-1} \prod_{p \equiv 3 \bmod 4} (1 - (-1)^m p^{-s})^{-2} \\
			&=\begin{cases}
				L(s,\xi^{2m}) \frac{\zeta(s) }{L(s,\chi_4)} (1-2^{-s}) &\text{ for } m \text{ even, } \\
				L(s,\xi^{2m}) \frac{L(s,\chi_4)}{\zeta(s) } (1-2^{-s})^{-1} &\text{ for } m \text{ odd. }
			\end{cases}
		\end{align*} 
		For each $m\geq 1$, $\ord_{s=1}(L(s,\xi^{2m})) = 0$, so, the function $L(s,\xi_m^{(2)})$ has a pole of order $1$ at $s=1$ when $m$ is even and a zero of order $1$ at $s=1$ when $m$ is odd.
	\end{proof}

\subsubsection{A family of characters for the twisted angles $\tilde{\theta}_p$}	

Following the same range of ideas as in the previous section, we now define a family of Hecke characters $(\psi_m)_{m\geq 0}$ that will be used in the proof of Theorem~\ref{Th race Gaussian primes A mod 4} and review the properties of their associated $L$-functions. 	
	To separate the congruence classes of~$a$, we use characters of larger modulus.

For each~$m\geq 0$  let~$\psi_m$ be the Hecke character on the multiplicative groups of fractional ideals of~$\mathbf{Z}[i]$ modulo~$(4)$ defined by
\begin{equation}\label{def psi_m}
	\psi_m((\alpha)) = \begin{cases}
		\big(\frac{\alpha}{\lvert \alpha \rvert}\big)^m &\text{ if } \alpha \equiv 1 \bmod (4) \\
		- \big(\frac{\alpha}{\lvert \alpha \rvert}\big)^m &\text{ if } \alpha \equiv 3+2i \bmod (4) \\
		0 &\text{ if } (\alpha, (4)) \neq 1.
	\end{cases}
\end{equation}
Then, for each $m\geq 0$, $\psi_m$ is a primitive unitary Hecke character of frequency~$m$, and conductor~$(4)$, its finite part is the Dirichlet character 
$$\psi_{m,\mathrm{fin}}: u \mapsto \begin{cases} u^{-m} &\text{ for } u \in\lbrace \pm 1, \pm i\rbrace \\  -\big(\frac{u}{3+2i}\big)^{-m} &\text{ for } u \in \lbrace 3+2i, -3-2i, -2 + 3i, 2 - 3i\rbrace
\end{cases}$$ 
representing the~$8$ invertible congruence classes modulo~$(4)$. 
We have, for any unramified splitting rational prime~$(p) = \mathfrak{p}\overline{\mathfrak{p}}$, and for any~$m \in \mathbf{N}$  
\begin{align}\label{Psi prop theta}
	\psi_{m}(\mathfrak{p}) + \psi_{m}(\overline{\mathfrak{p}}) =  \begin{cases}
		2\cos(m\theta_{p}) &\text{ if } p \equiv 1 \bmod 8 \\
		-2\cos(m\theta_{p}) &\text{ if } p \equiv 5 \bmod 8,
	\end{cases} 
\end{align}
which will be exactly what we need for Theorem~\ref{Th race Gaussian primes A mod 4}.

As in the previous section, the associated $L$-functions have the necessary properties to apply Theorem~\ref{Th_GeneDistLim}.
\begin{lem}\label{Lem L-fucntion Psi}
	For $m\geq 0$, the $L$-functions
$	L(s,\psi_{m}) $ seen as products over rational primes
	are analytic $L$-functions (in the sense of \cite[Def. 1.1]{DevinChebyshev}) of
	 degree~$d_m=2$ and analytic conductor~$\mathfrak{q}(\psi_m) = 
		64  (\tfrac{m}{2} +3)(\tfrac{m}{2} + 4)$.
\end{lem}

We then determine the sign of the functional equation of these $L$-functions.
\begin{lem}\label{Lem sign psi}
		The sign of the functional equation for $L(s,\psi_{m})$ depends on the congruence class of $m$ modulo $4$. 
	Precisely, one has
	\begin{align*}
		W(\psi_m) = \begin{cases}
			1 &\text{ if } m \equiv 0, 1, 2 \bmod 4, \\
			-1 &\text{ if } m \equiv 3 \bmod 4.  
		\end{cases}
	\end{align*}
	\end{lem}

\begin{proof}
Again using the formula \cite[(3.85)]{IK},
we have
\begin{align*}
	W(\psi_m) &= i^{-m} N(4)^{-\frac12} \psi_{m,\infty}(8) \sum_{x \in \mathbf{Z}[i]/(4)}\psi_{m,\mathrm{fin}}(x) e^{2i\pi \tr(\tfrac{x}{8})} \\
	&= \frac{i^{-m}}{4}  \sum_{x \in \lbrace \pm 1,\pm i \rbrace}x^{-m} (e^{2i\pi \tr(\tfrac{x}{8})} - e^{2i\pi \tr(\tfrac{x(3+2i)}{8})}) \\
	&= \frac{1}{2}( (1 - (-1)^m)i^{1-m} + (1+(-1)^m)) \\
	&=\begin{cases}
		1 &\text{ if } m\equiv 0 \bmod2, \text{ or } m \equiv 1 \bmod4, \\
		-1 &\text{ if } m \equiv 3 \bmod 4,
	\end{cases}
\end{align*}
which concludes the proof.
\end{proof}

Finally, we study the second moment $L$-function.
\begin{lem}\label{Lem 2nd at 1 psi}
	For $m \geq 0$ we have
$$\ord_{s=1}(L(s,\psi_m^{(2)})) = \begin{cases}
	-2 &\text{ for } m=0, \\
	-1 &\text{ for } m >0, \text{ even,}\\
	1 &\text{ for } m \text{ odd.}
\end{cases}$$

\end{lem}

\begin{proof}
	Observe that for all $m \geq 0$ one has $\psi_m(\mathfrak{p})^2 =   \xi_m(\mathfrak{p})^2$ for $\mathfrak{p}\mid p \equiv 1 \bmod 4$ and $\psi_m((p)) =   \xi_m((p))$ for $p \equiv 3 \bmod 4$.
		As $L$-functions over the rational primes, we can then deduce that, up to the factor at $2$, the second moments satify
	\begin{align*}
		L(s,\psi_{m}^{(2)}) = L(s,\xi_m^{(2)}). 
	\end{align*}
	We conclude that the orders of vanishing at $s=1$ are the same as in Lemma~\ref{Lem 2nd at 1 Xi}.
\end{proof}

	\subsection{Sign of the functional equation and vanishig at the central point}\label{subsection hypothese}
	
In the statements of Theorems~\ref{Th race Gaussian primes} and~\ref{Th race Gaussian primes A mod 4}, in order to give an explicit expression for the average values of the distributions $\mu_{\phi}$ and $\nu_{\phi}$ we make a bold assumption, namely that ``the vanishing at the central point of an Hecke $L$-function is exactly given by the sign of its functional equation''.
Before proceeding any further, let us discuss this assumption.

First, what is really meant by this assumption is the following expression for the order of vanishing at $\frac12$ for any $L$-function of Hecke character $\eta$ considered : 
$$\ord_{s= \frac12}(L(s,\eta)) = \frac{1 - W(\eta)}{2},$$
where $W(\eta)=\pm1$ is the sign of the functional equation of the $L$-function~$L(\cdot,\eta)$. 
In short, we think that the order of vanishing  is the smallest it can be (observe that $W(\eta) = -1$ in the functional equation of a real $L$-function forces vanishing at $\tfrac12$).

So precisely, by Lemmas~\ref{Lem Sign Xi} and~\ref{Lem sign psi} our assumptions are
\begin{equation}\label{Hyp Xi}
	\ord_{s= \frac12}(L(s,\xi_m)) = \begin{cases}
		0 &\text{ if } m \equiv 0, 1, 2, 3, 4, 6 \bmod 8, \\
		1 &\text{ if } m \equiv 5 , 7 \bmod 8, 
	\end{cases}
\end{equation}
and 
\begin{equation}\label{Hyp Psi}
	\ord_{s= \frac12}(L(s,\psi_m)) = \begin{cases}
		0 &\text{ if } m \equiv 0, 1, 2 \bmod 4, \\
		1 &\text{ if } m \equiv 3 \bmod 4. 
	\end{cases} 
\end{equation}

Such an assumption is in general believed to be true except for density-zero sets of $L$-functions.
One can see for example \cite{Greenberg} and \cite{Waxman,DDW} for partial results in families of $L$-functions associated to Hecke characters close to the ones we are considering in this paper.
However, there exists counter examples to the statement for all $L$-functions, in particular related to elliptic curves of large rank that are also not so far from the setting of this paper\footnote{As suggested by the referee, those could be interesting objects to consider if one wanted to find other potential complete biases.}, see for example \cite{Spe,FT}.

To give more grounding to our assumption, one can find our $L$-functions in the $L$-functions and Modular Forms Database \cite{lmfdb}, at least for $m\leq 24$, where the analytic rank is calculated and satisfies our assumption. They can be found as $L$-functions of modular forms : choosing level equal to $4N(\mathfrak{f}_m)$ (resp. $64$), weight $m+1$, bad $p$ exactly $2$, has CM with discriminant $-4$; one obtains a unique modular form which has the same $L$-function as $\xi_m$ (resp. $\psi_m$).
The assumption was also verified with the help of \texttt{PARI/GP}~\cite{PARI2} for $m \leq 2000$, see Appendix~\ref{App Pari} for details and code.

	\subsection{Proofs of Theorem~\ref{Th race Gaussian primes} and~\ref{Th race Gaussian primes A mod 4}}\label{subsection proof applications}
	
	We will need a technical lemma to compute the mean values.
	    
	\begin{lem}\label{Lem exp sum}
		Let $q,a \in \mathbf{N}$, $q>0$ and $t \notin \frac{2\pi}{q}\mathbf{Z}$
		One has 
		\begin{multline*}
			\sum_{m = 0}^{N}(e^{i(qm +a)t} + e^{i(-qm -a)t}) = \frac{\sin((\tfrac{q}{2}(2N+1)+a)t) - \sin((a-\tfrac{q}{2})t)}{\sin(\tfrac{q}{2}t)}
		\end{multline*}
	\end{lem}

	\begin{proof}
		For $t \notin \frac{2\pi}{q}\mathbf{Z}$ we can sum the geometric sums
		\begin{align*}
			\sum_{m = 0}^{N}(e^{i(qm +a)t} + e^{i(-qm -a)t}) &= e^{iat} \frac{ 1 - e^{iq(N+1)t}}{1 - e^{iqt}} + e^{-iat} \frac{ 1 - e^{-iq(N+1)t}}{1 - e^{-iqt}} \\
			&= 2\cos((\tfrac{q}{2}N + a)t) \frac{\sin(\tfrac{q}{2}(N+1)t)}{\sin(\tfrac{q}{2}t)} \\
			&= \frac{\sin((\tfrac{q}{2}(2N+1)+a)t) - \sin((a-\tfrac{q}{2})t)}{\sin(\tfrac{q}{2}t)},
		\end{align*}
		this is the statement of the Lemma.
	\end{proof}
	
	We can now finally write the proofs of our applications of Theorem~\ref{Th_GeneDistLim}.

    \begin{proof}[Proof of Theorem~\ref{Th race Gaussian primes} as a consequence of Theorem~\ref{Th_GeneDistLim}]
    	
    	Since $\phi$ is even of period $2\pi$ with converging Fourier series, one can write $\phi$ as a sum of its Fourier series, $\phi(\theta) = \sum_{m\geq 0} c_{m}(\phi) \cos (m\theta)$
    	where the Fourier coefficients are
    	\begin{align*}
    	c_0(\phi) &=  \frac{1}{2\pi}\int_{-\pi}^{\pi} \phi(t) \diff t\\
    	c_m(\phi) &= \frac{1}{\pi}\int_{-\pi}^{\pi} \phi(t)\cos(mt) \diff t \text{, for } m\geq 1.
    	\end{align*}
    	
    		We apply Theorem~\ref{Th_GeneDistLim} with $\mathcal{S} = \mathcal{S}_1 = \lbrace L(s,\xi_{m}) : m\geq 0 \rbrace$, and $\underline{c} = \lbrace c_m(\phi): m\geq 0\rbrace$.
    	Lemma~\ref{Lem L-func Xi} implies that the hypotheses of Theorem~\ref{Th_GeneDistLim} are satisfied for this set and 
    	for $m\geq 1$ we have 
    		$\mathfrak{q}(\xi^m) \ll m^2$. 
   	The hypothesis on the Fourier coefficients of $\phi$ is exactly here to ensure that the series $\sum_{m\geq 1} \lvert c_m(\phi) \rvert \log m^2$ is convergent.

    	 Thus, under the Riemann Hypothesis for $L(s,\xi_m)$, $m\geq 0$, the function
    	\begin{align*}
    	E_{\phi}(x) &=\frac{\log x}{2\sqrt{x}}\Bigg(
    	\sum_{\substack{p\leq x \\ p\equiv 1 \bmod 4}}\sum_{m\geq 0}c_{m}(\phi) 2\cos(m\theta_p) + \sum_{m\geq 0}c_{m}(\phi)\ord_{s=1}(L(s,\xi_m)) \Li(x) \Bigg) \\
    	&=  \frac{\log x}{\sqrt{x}}\Bigg(
    	\sum_{\substack{p\leq x \\ p\equiv 1 \bmod 4}} \phi(\theta_p)  - \frac{c_{0}(\phi) \Li(x)}{2} \Bigg)   	
    	\end{align*}
    	admits a limiting distribution $\mu_{\phi}$
    	(where we used the fact that for all $m\geq 1$, the function $L(\cdot,\xi_m)$ does not have a pole nor a zero at $s=1$).	
    	Moreover, Theorem~\ref{Th_GeneDistLim} yields the following expression for the mean value of $\mu_{\phi}$:
    	\begin{align*}
    	\mathbb{E}(\mu_{\phi}) = \frac{1}{2}\sum_{m\geq 0}c_{m}(\phi)
    	\left(\ord_{s=1}(L(s,\xi_m^{(2)}))  - 2 \ord_{s= \frac12}(L(s,\xi_{m})) \right).
    	\end{align*}
      Lemma~\ref{Lem 2nd at 1 Xi} gives
    	\begin{align*}
    	\sum_{m\geq 0}c_{m}(\phi)\ord_{s=1}(L(s,\xi_{m}^{(2)})) = -c_0(\phi) - \sum_{m\geq 0}(-1)^mc_{m}(\phi) =  -\hat{\phi}(0) - \phi(\pi),
    	\end{align*}  
  while Assumption~\eqref{Hyp Xi} yields
    	\begin{align*}
    	\sum_{m\geq 0}c_{m}(\phi) \ord_{s= \frac12}(L(s,\xi_{m}))  = \sum_{\substack{m \geq 0 \\ m\equiv 5,7 \bmod 8}} c_m(\phi).
    	\end{align*}
    We conclude the proof by writing this sum as an integral. 
    Using Lemma~\ref{Lem exp sum}, we have for smooth $\phi$ supported outside $\frac{\pi}{4}\mathbf{Z}$, 
    \begin{align*}
    	\sum_{\substack{m \geq 0 \\ m \equiv 5, 7 \bmod{8}}}(\hat{\phi}(m) + \hat{\phi}(-m)) 
    	&= \lim_{N\rightarrow\infty} \frac{1}{2\pi} \int_{-\pi}^{\pi} \phi(t)2\frac{\cos(t)}{\sin(4t)}(\sin((8N+10)t) - \sin(2t)) \diff t \\
    	&= -\frac{1}{2\pi} \int_{-\pi}^{\pi} \phi(t)\frac{\cos(t)}{\cos(2t)} \diff t.
    \end{align*}
    Moreover, the sum at $t \in \frac{\pi}{4}\mathbf{Z}$ gives the value $\frac{\phi(0) - \phi(\pi)}{4}$ which concludes the proof.
     \end{proof}

	The proof of Theorem~\ref{Th race Gaussian primes A mod 4} follows essentially the same lines. 
	\begin{proof}[Proof of Theorem~\ref{Th race Gaussian primes A mod 4} as a consequence of Theorem~\ref{Th_GeneDistLim}]
	Recall that by definition of~$\psi_m$ and~\eqref{Psi prop theta}, one has, for $\phi$ even and $2\pi$-periodic, 
	$$\phi(\theta_{p})\mathbf{1}_{p\equiv 1 \bmod 8} - \phi(\theta_{p})\mathbf{1}_{p\equiv 5 \bmod 8} = \sum_{m\geq 0} c_m(\phi) \big( \psi_{m}(\mathfrak{p}) + \psi_{m}(\overline{\mathfrak{p}}) \big),$$
	where the $c_m(\phi)$'s are again $\phi$'s Fourier coefficients.

	We apply Theorem~\ref{Th_GeneDistLim} with $\mathcal{S} = \mathcal{S}_2 = \lbrace L(s,\psi_{m}) : m\geq 0\rbrace$, where by Lemma~\ref{Lem L-fucntion Psi}, each of these $L$-function has degree $d_m =2 $, and conductor $\mathfrak{q}(\psi_m) \ll (m+1)^2$ for $m\geq 0$, finally we take $c_m = c_m(\phi)$ for $m\geq 0$.
	This gives the existence of the limiting logarithmic distribution $\nu_{\phi}$ under the Riemann Hypothesis for the $L(s,\psi_{m}), m\geq 0$.
	Moreover, its mean value is given by
	\begin{align*}
		\mathbb{E}(\nu_{\phi}) = \frac{1}{2}\sum_{m\geq 0}c_{m}(\phi)
		\left(\ord_{s=1}(L(s,\psi_m^{(2)}))  - 2 \ord_{s= \frac12}(L(s,\psi_{m})) \right). 
	\end{align*}
	Lemma~\ref{Lem 2nd at 1 psi} gives
\begin{align*}
	\sum_{m\geq 0}c_{m}(\phi)\ord_{s=1}(L(s,\psi^{m(2)})) = -c_0(\phi) - \sum_{m\geq 0}(-1)^mc_{m}(\phi) =  -\hat{\phi}(0) - \phi(\pi).
\end{align*}
To evaluate the contribution of the zeros at $s=\frac12$ under Assumption~\eqref{Hyp Psi},
we apply Lemma~\ref{Lem exp sum} to obtain
for smooth $\phi$ supported outside $\frac{\pi}{2}\mathbf{Z}$,
\begin{align*}
	\sum_{\substack{m \geq 0 \\ m \equiv 3 \bmod{4}}}(\hat{\phi}(m) + \hat{\phi}(-m)) &= \lim_{N\rightarrow\infty} \frac{1}{2\pi} \int_{-\pi}^{\pi} \phi(t)\frac{\sin((4N+5)t) - \sin(t)}{\sin(2t)} \diff t \\
	&= -\frac{1}{2\pi} \int_{-\pi}^{\pi} \phi(t)\frac{1}{2\cos(t)} \diff t.
\end{align*}
Then considering the sum at $t \in \frac{\pi}{2}\mathbf{Z}$, gives the value $\frac{\phi(0) - \phi(\pi)}{4}$ which concludes the proof of Theorem~\ref{Th race Gaussian primes A mod 4}.
    \end{proof}

	\section{Heuristic for the conjectures}\label{Sec Heuristic}
	
	Let us now develop our heuristic argument for Conjecture~\ref{Conj bias even odd} and~\ref{Conj bias A mod 4}. 
	The main idea is to replace $\phi$ in Theorems~\ref{Th race Gaussian primes} and~\ref{Th race Gaussian primes A mod 4} by a difference of indicator functions.
	Precisely, we take $\phi_1 = \mathbf{1}_{[0,\tfrac{\pi}{4}]\cup[\frac{3\pi}{4},\pi]} - \mathbf{1}_{(\tfrac{\pi}{4},\frac{3\pi}{4})}$ in Theorem~\ref{Th race Gaussian primes} and $\phi_2 = \mathbf{1}_{[0,\tfrac{\pi}{2}]} - \mathbf{1}_{(\tfrac{\pi}{2},\pi]}$ in Theorem~\ref{Th race Gaussian primes A mod 4}, where the functions $\phi_i$ are defined on $[0,\pi]$ and we extend their definition to  $\mathbf{R}$ so that they are even and $2\pi$-periodic. 
	Then, we have 
	$$c_m(\phi_1) = \begin{cases}
		\frac{8}{m\pi} &\text{if } m \equiv 2 \bmod8 \\
		-\frac{8}{m\pi} &\text{if } m \equiv 6 \bmod8, \\
		0 &\text{otherwise}
	\end{cases}
\text{ and } 
c_m(\phi_2) = \begin{cases}
	\frac{4}{m\pi} &\text{if } m \equiv 1 \bmod4 \\
	-\frac{4}{m\pi} &\text{if } m \equiv 3 \bmod4 \\
	0 &\text{otherwise.}
\end{cases}$$
As observed earlier, these Fourier coefficients do not satisfy the hypothesis of decay needed to apply Theorem~\ref{Th race Gaussian primes} or~\ref{Th race Gaussian primes A mod 4}.
Let us however continue the heuristic argument by ignoring this.
	
	Then we would obtain logarithmic limiting distributions $\mu_{\phi_1}$ and $\nu_{\phi_2}$ with mean values equal to 
	\begin{align*}
		\mathbb{E}(\mu_{\phi_1}) &=  - \frac{\phi_1(0) + \phi_1(\pi)}{4} +\frac{1}{2\pi} \int_{-\pi}^{\pi} \phi_1(t)\frac{\cos(t)}{\cos(2t)} \diff t = \frac{-1}{2}\\
\text{and }	\mathbb{E}(\nu_{\phi_2}) &=  -\frac{\phi_2(0) + \phi_2(\pi)}{4} +\frac{1}{2\pi} \int_{-\pi}^{\pi} \phi_2(t)\frac{1}{2\cos(t)} \diff t = \infty.
\end{align*}
The variance of $\mu_{\phi_1}$ is given by 
\[\Var(\mu_{\phi_1}) = 2\sum_{\gamma \in \mathcal{Z}_{\mathcal{S}_1,\underline{c}}}  \frac{\lvert \sum_{m \geq 0}c_m(\phi_1)\ord(\gamma,m)\rvert^{2}}{\frac14 +\gamma^{2}} \in \mathbf{R}\cup\lbrace \infty\rbrace, \]
where $\mathcal{S}_1 = \lbrace L(s,\xi_{m}) : m\geq 0 \rbrace$ is the set of Hecke $L$-functions used in the proof of Theorem~\ref{Th race Gaussian primes}, and $\underline{c}= \lbrace c_m(\phi_1) : m\geq 0\rbrace$. 
We obtain the same formula for $\Var(\nu_{\phi_2})$ with $\phi_1$ replaced by $\phi_2$ and $\mathcal{S}_1$ replaced by $\mathcal{S}_2= \lbrace L(s,\psi_{m}) : m\geq 0\rbrace$, the set of Hecke $L$-functions used in the proof of Theorem~\ref{Th race Gaussian primes A mod 4}. 

Let $\phi = \phi_1$ or $\phi_2$ and $\mathcal{S} = \mathcal{S}_1$ or $\mathcal{S}_2$. Let us assume that there exists $B>0$ such that for all $\gamma\in \mathcal{Z}_{\mathcal{S},\underline{c}}$, we have $\lvert \sum_{m \geq 0}c_m(\phi)\ord(\gamma,m)\rvert < B \max\lbrace \lvert c_m(\phi)\ord(\gamma,m) \rvert : m\geq 0 \rbrace$.
That is, we assume that for each $\gamma >0$ there are not to many $L$-functions in $\mathcal{S}$ that vanish at $\tfrac12 + i \gamma$. This hypothesis is reminiscent of the bounded multiplicity hypothesis used by Fiorilli in \cite{Fiorilli_EC} and is supported by the general idea that zeros of $L$-functions should be independent while being weaker than the Linear Independence hypothesis.
Then we have
\begin{align*}
	\Var(\mu_{\phi_1}) &< B^2 \sum_{m \geq 0}\sum_{\gamma \in \mathcal{Z}_{m}}  \frac{\lvert c_m(\phi)\ord(\gamma,m)\rvert^{2}}{\frac14 +\gamma^{2}} \\
	&= B^2 \sum_{m \geq 0} \lvert c_m(\phi)\rvert^{2}\sum_{\gamma \in \mathcal{Z}_{m}}  \frac{\lvert \ord(\gamma,m)\rvert^{2}}{\frac14 +\gamma^{2}}.
\end{align*}
Recall that $$\sum_{\gamma \in \mathcal{Z}_{m}}  \frac{\lvert \ord(\gamma,m)\rvert^{2}}{\frac14 +\gamma^{2}} \ll (\log \mathfrak{q}_m)^3 \ll (\log m)^3,$$
and that $c_m(\phi) \ll \frac{1}{m}$.
It yields
$\Var(\mu_{\phi_1}) < \infty$ and similarly, $\Var(\nu_{\phi_2}) < \infty$.

This concludes our heuristic for Conjecture~\ref{Conj bias even odd} : we found a limiting logarithmic distribution with negative mean value and bounded variance. This indicates a bias towards negative values in the distribution of the values of the function $D_1$.
Moreover, if we assume that the set $\mathcal{Z}_{\mathcal{S}_1,\underline{c}}$ has many self-sufficient elements, as in Proposition~\ref{Prop inclusive} --- assumption that is again supported by the fact that the zeros of $L$-functions should be independent --- then we deduce the Omega-result with the help of Corollary~\ref{Cor Omega result}.

In the case of Conjecture~\ref{Conj bias A mod 4}, we obtained a limiting logarithmic distribution that has infinite mean value and bounded variance, this indicates a very strong bias in the direction of positive values.
Let us approach this heuristic by another way and
let us write $\phi_{2,N}(\theta) = \sum_{m \leq N} c_m(\phi_2) \cos(m\theta)$.
By Chebyshev's inequality, as in \cite[Lem. 2.10]{Fiorilli_HighlyBiased} (see also \cite[Cor. 5.8]{DevinChebyshev}) 
we have
$$ \nu_{\phi_{2,N}}([0,\infty)) \geq 1 - \frac{\Var(\mu_{\phi_{2,N}})}{\mathbb{E}(\nu_{\phi_{2,N}})} = 1 - O((\log N)^{-1}). $$
This heuristically indicates that, in the limit when $N\rightarrow\infty$ there is a complete bias, namely we expect $ \nu_{\phi_{2}}([0,\infty))  =1$, or in other terms, the function $D_2$ is almost always (in logarithmic scale) positive.

\section{Proof of Theorem~\ref{Th_GeneDistLim}}\label{Section proof theo general}
	
		To prove Theorem~\ref{Th_GeneDistLim}, we follow the proof of \cite[Th. 1.2]{ANS} or \cite[Th. 2.1]{DevinChebyshev}, but keeping explicit the dependency on the $L$-function.
We show that Theorem~\ref{Th_GeneDistLim} is the consequence of the following result.

	\begin{prop}\label{Prop_LimOfDist}
		Under the hypotheses of Theorem~\ref{Th_GeneDistLim},
		for each $M >0$ and $T>2$, let
		\[G_{\mathcal{S},\underline{c},M,T}(x) = m_{\mathcal{S},\underline{c}} -\sum_{\gamma\in\mathcal{Z}_{\mathcal{S},\underline{c},M}(T)}2\re\left(\ord_{\mathcal{S},\underline{c}}(\gamma)\frac{x^{i\gamma}}{\frac12 + i\gamma}\right).\]
			The function $G_{\mathcal{S},\underline{c},M,T}$ admits a limiting logarithmic distribution $\mu_{\mathcal{S},\underline{c},M,T}$.
		Moreover, there exists a function $M(T)$ satisfying $M(T)\rightarrow\infty$ as $T\rightarrow\infty$ such as for any bounded Lipschitz continuous function $g$, one has
		\begin{align*}
		\lim_{T\rightarrow\infty}\int_{\mathbf{R}}g(t)\diff\mu_{\mathcal{S},\underline{c},M(T),T}(t) = 
		\int_{\mathbf{R}}g(t)\diff\mu_{\mathcal{S},\underline{c}}(t).
		\end{align*}
	\end{prop}

	For every $M,T$ fixed, the set $\mathcal{Z}_{\mathcal{S},\underline{c},M}(T)$ is finite.
	Thus, by hypothesis, the function $G_{\mathcal{S},\underline{c},M,T}$ is well-defined and it admits a limiting logarithmic distribution as a consequence of Kronecker--Weyl equidistribution Theorem (see e.g. \cite[Th. 4.2]{DevinChebyshev}, \cite[Lem. 4.3]{Humphries}, or \cite[Lem. B.3]{MartinNg}).
	The convergence of the measures needs more work on the estimation of the error terms. 
	Let us first recall the following precise form of \cite[Prop. 4.2]{ANS}, \cite[(4.5)]{DevinChebyshev}. 
	
	\begin{prop}\label{Prop for one L function}
		Let $L(f,s)$ be an analytic $L$-function of degree $d$, we denote by~$L(f^{(2)},s)$ its second moment $L$-function.  
		Assume the Riemann Hypothesis holds for $L(f,s)$ and $L(f^{(2)},s)$ . 
		Let $T >0$, and 
		\begin{align*}
		G_{f,T}(x) = m_{f} -\sum_{\gamma\in\mathcal{Z}_{f}(T)}2\re\left(\ord(\gamma, L(f,s))\frac{x^{i\gamma}}{\frac12 + i\gamma}\right).
		\end{align*}
		We have the following estimate for all $x>0$ :
		\begin{align*}
		E_f(x) :=&\frac{\log x}{\sqrt{x}}\left(\sum_{p\leq x} \lambda_{f}(p) + \ord_{s=1}(L(f,s))\Li(x)\right) \\
		&= G_{f,T}(x) 
		- \epsilon_{f}(x,T)
		+ O\left(d\frac{\log\mathfrak{q}(f)}{\log x} \right)
		\end{align*}
		where the function $\epsilon_{f}(x,T)$ satisfies
		\begin{equation}\label{Bound_epsilon}
		\int_{2}^{Y}\lvert \epsilon_{f}(e^{y},T)\rvert^{2} \diff y \ll Y\frac{\left(\log(\mathfrak{q}(f)T^d)\right)^2}{T} + \frac{\left(\log(\mathfrak{q}(f)T^d)\right)^2\log T}{T}
		\end{equation}
		with an absolute implicit constant.
	\end{prop}
	
	\begin{proof}
		The proof is contained in \cite{DevinChebyshev}, where the dependency in the $L$-function is not always written explicitly.
		In particular from \cite[Prop. 4.4]{DevinChebyshev} we have 
		\begin{align*}
		\psi(f,x) + \ord_{s=1}(L(f,s))x  = 
		- \sum_{\substack {L(f,\rho)=0 \\ \lvert\im(\rho)\rvert\leq T}}\frac{x^{\rho}}{\rho} - x^{\frac12}\epsilon_{f}(x,T) + 
		O\left(\log(\mathfrak{q}(f)x^d)\log x \right)
		\end{align*}
		with an absolute implicit constant. 
		Then taking care of the sum over squares of primes, we
		use the Ramanujan--Petersson Conjecture and the Prime Number Theorem to obtain:
		\[\theta(f,x) :=\sum_{p\leq x}\lambda_{f}(p)\log p = \psi(f,x) 
		- \sum_{p^{2}\leq x}\left(\sum_{j=1}^{d} \alpha_{j}(p)^2\right) \log p
		+ O(dx^{\frac13}).\]
		To evaluate the second term, we use the Riemann Hypothesis for the function $L(f^{(2)},s) = L(\Sym^2f,s) L(\wedge^2f,s)^{-1}$.
		One has
		$\frac{L'(f^{(2)},s)}{L(f^{(2)},s)} = \frac{L'(\Sym^2f,s)}{L(\Sym^2f,s)}  - \frac{L'(\wedge^2f,s)}{L(\wedge^2f,s)}$, thus 
		\begin{align*}
		\sum_{p^{2}\leq x}\left(\sum_{j=1}^{d} \alpha_{j}(p)^2\right) \log p 
		&= -\ord_{s=1}(L(\Sym^2f,s))x^{\frac{1}{2}}  + O(x^{\frac{1}{4}}\log x \log (x^{d(d+1)/2}\mathfrak{q}(\Sym^2f))) \\
		&\quad + \ord_{s=1}(L(\wedge^2f,s))x^{\frac{1}{2}}  + O(x^{\frac{1}{4}}\log x \log (x^{d(d-1)/2}\mathfrak{q}(\wedge^2f))) + O(dx^{\frac14}) \\
		&= -\ord_{s=1}(L(f^{(2)},s))x^{\frac{1}{2}} + O(dx^{\frac{1}{4}}\log x \log(x^{d}\mathfrak{q}(f)) ).	
		\end{align*}
	Finally, using Stieltjes integral, we write $E_{f}(x) = \frac{\log x}{x^{\frac12}} \int_{2}^{x} \frac{\diff (\theta(f,t) + \ord_{s=1}(L(f,s))t)}{\log t}$.
	After integration by parts this yields
	\begin{multline*}
	E_{f}(x) =  \frac{1}{\sqrt{x}}\big(\psi(f,x) + x\ord_{s=1}(L(f,s))\big) + \ord_{s=1}(L(f^{(2)},s)) \\ +
	O\left( \frac{\log x}{\sqrt{x}}  \int_{2}^{x} \frac{\psi(f,t) + t\ord_{s=1}(L(f,s)) + \sqrt{t}\ord_{s=1}(L(f^{(2)},s))}{t(\log t)^{2}}\diff t \right) \\
	+ O(dx^{-\frac16}\log x  \log(\mathfrak{q}(f)x^d)). 
	\end{multline*}
	Using the explicit formula
	\begin{multline*}
	\psi(f,x) + \ord_{s=1}(L(f,s))x  = -\sum_{\substack {L(f,\rho)=0 \\ \lvert\im(\rho)\rvert\leq X}}\frac{x^{\rho}}{\rho}\\   + 
	O\left(d\log x + \frac{x}{X}\left(d(\log x)^2 + \log(\mathfrak{q}(f)X^{d})\right) + \log(\mathfrak{q}(f)X^{d})\log X\right), 
	\end{multline*}
	and another integration by parts to evaluate the second term we have
	\begin{multline*}
	 \int_{2}^{x} \frac{\psi(f,t) + t\ord_{s=1}(L(f,s)) + \sqrt{t}\ord_{s=1}(L(f^{(2)},s))}{t(\log t)^{2}}\diff t \\
	 \ll \ord_{s=1}(L(f^{(2)},s))\frac{x^{\frac12}}{(\log x)^2}  +  \sum_{\substack {L(f,\rho)=0 \\ \lvert\im(\rho)\rvert\leq x}}\frac{\lvert x^{\rho}\rvert }{\lvert\rho^2\rvert (\log x)^2} + \log(\mathfrak{q}(f)x^{d})\log x
	\end{multline*}
	where after the integration we take $X=x$.
	The sum over the zeros is convergent and this concludes the proof of Proposition~\ref{Prop for one L function}.
	\end{proof}
	
Then we sum over the $L$-functions, we obtain the following result.

\begin{prop}\label{Prop sum L functions}
	Under the hypotheses of Theorem~\ref{Th_GeneDistLim},
	there exists a function $M(T) = M_{\mathcal{S},\underline{c}}(T)$, with $M(T) \rightarrow\infty$ as $T \rightarrow\infty$, 
	such that
	we have the following estimate for all $x>0$ and $T>0$.
	\begin{equation*}
	E_{\mathcal{S},\underline{c}}(x)= G_{\mathcal{S},\underline{c},M(T),T}(x) - \epsilon_{\mathcal{S},\underline{c}}(x,T) + O_{\mathcal{S},\underline{c}}\left( \frac{1}{\log x} + \frac{1}{T}\right),
	\end{equation*}
	where the function $\epsilon_{\mathcal{S},\underline{c}}(x,T)$ satisfies
	\begin{equation}\label{Bound_epsilon total}
	\int_{2}^{Y}\lvert \epsilon_{\mathcal{S},\underline{c}}(e^{y},T)\rvert^2 \diff y \ll_{\mathcal{S},\underline{c}} Y\frac{(\log T)^2}{T} + \frac{(\log T)^{3}}{T}.
	\end{equation}
\end{prop}

\begin{proof}[Proof of Proposition~\ref{Prop sum L functions}]
	By definition, one has
	\begin{equation*}
	E_{\mathcal{S},\underline{c}}(x) = \sum_{m\geq 0} c_{m}E_{f_m}(x).
	\end{equation*}
Thus, using Proposition~\ref{Prop for one L function}, one has
\begin{align*}
E_{\mathcal{S},\underline{c}}(x) &= \sum_{m = 0}^{\infty} c_{m} \left(G_{f_m,T}(x) 
- \epsilon_{f_m}(x,T)
+ O\left(d_m\frac{\log\mathfrak{q}(f_m)}{\log x} \right) \right).
\end{align*}
For each $x$ and $T$, the three series are convergent, we separate
\begin{multline*}
	E_{\mathcal{S},\underline{c}}(x)= G_{\mathcal{S},\underline{c},M,T}(x) - 
	\sum_{m >M} \Bigg(c_m \sum_{\substack{\gamma\in\mathcal{Z}_{f_m}(T) \\ \gamma\notin\mathcal{Z}_{\mathcal{S},\underline{c},M}(T) }}2 \re\Big( \ord(\gamma,m) \frac{x^{i\gamma}}{\frac12 + i\gamma}\Big) \Bigg) \\
	 - \sum_{m = 0}^{\infty}c_{m}\epsilon_{f_m}(x,T)
	+ O\left(\sum_{m = 0}^{\infty}\lvert c_{m}\rvert d_m\frac{\log\mathfrak{q}(f_m)}{\log x} \right).
\end{multline*}
For each $m>M$, one has
\begin{align*}
\sum_{\gamma\in\mathcal{Z}_{f_m}(T)\smallsetminus\mathcal{Z}_{\mathcal{S},\underline{c},M}(T)}2 \re\left( \ord(\gamma,m) \frac{x^{i\gamma}}{\frac12 + i\gamma}\right)
\ll \log T \log(\mathfrak{q}(f_m)T). 
\end{align*}
Thus
\begin{align*}
\sum_{m >M} c_m \sum_{\substack{\gamma\in\mathcal{Z}_{f_m}(T) \\ \gamma\notin\mathcal{Z}_{\mathcal{S},\underline{c},M}(T) }}2 \re\left( \ord(\gamma,m) \frac{x^{i\gamma}}{\frac12 + i\gamma}\right)
\ll (\log T)^{2}  \sum_{m >M} \lvert c_m\rvert \log(\mathfrak{q}(f_m)).
\end{align*}
For each $T$ the series is convergent, so there exist $M = M_{\mathcal{S},\underline{c}}(T)$ such that 
\begin{equation*}
 \sum_{m >M} \lvert c_m\rvert  \log(\mathfrak{q}(f_m)) \leq \frac{1}{(\log T)^{2}T}.
\end{equation*}
Let
\begin{align*}
\epsilon_{\mathcal{S},\underline{c}}(x,T) =  \sum_{m = 0}^{\infty}c_{m}\epsilon_{f_m}(x,T).
\end{align*}
One has
\begin{align}
\label{Eq sum epsilon}
\nonumber\int_{2}^{Y} \lvert \epsilon_{\mathcal{S},\underline{c}}(e^y,T) \rvert^2 \diff y &\leq 
\sum_{m = 0}^{\infty}\sum_{n =0}^{\infty} \lvert c_{m} \rvert \cdot \lvert c_n\rvert  \int_{2}^{Y} \lvert \epsilon_{f_m}(e^y,T) \rvert \cdot  \lvert \epsilon_{f_n}(e^y,T) \rvert \diff y \\
\nonumber
&\ll \left\lbrace \sum_{m = 0}^{\infty} \lvert c_{m} \rvert
\left(Y\frac{\left(\log(\mathfrak{q}(f_m)T^{d_m})\right)^2}{T} + \frac{\left(\log(\mathfrak{q}(f_m)T^{d_m})\right)^2\log T}{T}\right)^{\frac12} \right\rbrace^2 \\
&\ll_{\mathcal{S},\underline{c}}  Y\frac{(\log T)^2}{T} + \frac{(\log T)^{3}}{T}. 
\end{align}
Finally, since the series
\begin{align*}
 \sum_{m \geq 0}\lvert c_{m}\rvert d_m\log\mathfrak{q}(f_m) 
\end{align*}
is convergent, the proof is complete.
\end{proof}

We can now come back to the proof of Proposition~\ref{Prop_LimOfDist}.

\begin{proof}[Proof of Proposition~\ref{Prop_LimOfDist}]
By Proposition~\ref{Prop sum L functions},  $E_{\mathcal{S},\underline{c}}$ is a $B^2$-almost periodic function well approximated by the $G_{\mathcal{S},\underline{c},M(T),T}$'s.
Thus Proposition~\ref{Prop_LimOfDist} follows from\footnote{Note that there is a misprint in the proof of \cite[Th.~2.9]{ANS}, (2.10) should read $\frac{1}{Y}\int_{0}^{Y}\lvert \vec{\phi}(y) - \vec{P}_M(y)\rvert \diff y < \epsilon$ for~$Y$ large enough, the constant $A_\epsilon$ may have to be enlarged to include smaller~$Y$'s, see also \cite[Th. 1.17]{Bailleul_Kronecker}, correcting this in more details.} \cite[Th.~2.9]{ANS}. 
\end{proof}
	
Then Theorem~\ref{Th_GeneDistLim} follows.
\begin{proof}[Proof of Theorem~\ref{Th_GeneDistLim}]
	The existence of the limiting logarithmic distribution $\mu_{\mathcal{S},\underline{c}}$ is stated in Proposition~\ref{Prop_LimOfDist}. 
	In the process of the proof, we used the fact that the function $E_{\mathcal{S},\underline{c}}$ is a $B^2$-almost periodic function, by \cite[Chap. II, \S 6, 4°]{Besicovitch55} it admits a mean value which is
		\[\mathbb{E}(\mu_{\mathcal{S},\underline{c}}) = \lim_{T\rightarrow\infty}\mathbb{E}(\mu_{\mathcal{S},\underline{c},M(T),T}) =  m_{\mathcal{S},\underline{c}}.\] 
	Then it follows from \cite[Chap. II, \S 9, 1°]{Besicovitch55} that it admits a second moment which is given via Parseval's identity.
	The formula for the variance follows as in \cite[Th. 1.14]{ANS} and \cite[Lem. 2.5, 2.6]{Fiorilli_HighlyBiased}.
		
	For the decay of the tails of the distribution, the proof is similar to the proof of \cite[Th. 1.2]{RS} (see also \cite[Lem. 4.8]{DevinChebyshev}) noting that the measure $\mu_{\mathcal{S},\underline{c},M(T),T}$ is supported inside an interval of the form $[-A(\log T)^2, A(\log T)^2]$, for a positive constant $A$ depending on $\mathcal{S}$ and $\underline{c}$.		
	\end{proof}

Let us now prove the results on the support of $\mu_{\mathcal{S},\underline{c}}$ and on sign changes that depend on supplementary conditions.
	
\begin{proof}[Proof of Proposition~\ref{Prop lower bound tails}] 
	The proof is similar to the proof of~\cite[Th.~1.2]{RS}, Following the notation of~\cite[Sec. 2.2]{RS},
	let $\epsilon >0$, $t \geq \log2 + \tfrac12\epsilon$, and
	$$F_{\epsilon}(t) = \frac{1}{\epsilon} \int_{t -\frac{\epsilon}{2}}^{t +\frac{\epsilon}{2}} E_{\mathcal{S},\underline{c}}(e^y)\diff y.$$
	Using Proposition~\ref{Prop sum L functions} and the bound 
	\begin{align*}
		\lvert \epsilon_{\mathcal{S},\underline{c}}(x,T)\rvert \leq   \sum_{m = 0}^{\infty}\lvert c_{m}\epsilon_{f_m}(x,T)\rvert
		\ll \frac{\log x}{\sqrt{x}} + \frac{\sqrt{x}}{T}\Big((\log x)^2 + \log T\Big), 
	\end{align*}
letting $T\rightarrow\infty$, we have
\begin{align*}
	F_{\epsilon}(t) = \frac{4}{\epsilon}\sum_{\gamma \in \mathcal{Z}_{\mathcal{S},\underline{c}}}\ord_{\mathcal{S},\underline{c}}(\gamma)\frac{\sin (t\gamma) \sin (\frac{\epsilon}{2}\gamma)}{\gamma^2} + O(1).
\end{align*}
The sum $\sum_{\gamma \in \mathcal{Z}_{\mathcal{S},\underline{c}}}\frac{\lvert \ord_{\mathcal{S},\underline{c}}(\gamma)\rvert}{\gamma^2}$ converges, so there exists $T= T(\epsilon)$ such that
the function 
\begin{align*}
	\tilde{F}_{\epsilon}(t) = \frac{4}{\epsilon}\sum_{\gamma \in \mathcal{Z}_{\mathcal{S},\underline{c},M(T)}(T)}\ord_{\mathcal{S},\underline{c}}(\gamma)\frac{\sin (t\gamma) \sin (\frac{\epsilon}{2}\gamma)}{\gamma^2}
\end{align*}
satisfies
\begin{align*}
	F_{\epsilon}(t) = \tilde{F}_{\epsilon}(t)  + O(1).
\end{align*}
It is then enough to show that $\tilde{F}_{\epsilon}(t)$ is large on a large set. As this is a finite sum, the proof follows from the same argument as in~\cite[Sec. 2.2]{RS}, under the condition that $\ord_{\mathcal{S},\underline{c}}(\gamma)\geq 0$ for all $\gamma$.
\end{proof}

\begin{proof}[Proof of Proposition~\ref{Prop inclusive}]
	The proof is similar to the proof of~\cite[Th.~1.5(c)]{MartinNg}.
	Using~\cite[Lem.~3.8 and Prop.~3.10]{MartinNg}, we write 
	$\mu_{\mathcal{S},\underline{c}} = \mu^{\mathrm{LI}}\ast\mu^{N}$
	where $\hat{\mu}^{\mathrm{LI}}(\xi) = \prod_{\gamma \in \mathcal{Z}_{\mathcal{S},\underline{c}}^{\mathrm{LI}}} J_0\Bigg(\Big\lvert\frac{ 2 \ord_{\mathcal{S},\underline{c}}(\gamma)\xi}{\frac{1}{2} + i \gamma}\Big\rvert\Bigg)$ and $\mu^{N}$ has positive mass in a small interval centred at~$0$.
	In particular the law of $\mu^{\mathrm{LI}}$ is the same as the law of $\sum_{\gamma \in \mathcal{Z}_{\mathcal{S},\underline{c}}^{\mathrm{LI}}} \Big\lvert\frac{ 2 \ord_{\mathcal{S},\underline{c}}(\gamma)}{\frac{1}{2} + i \gamma}\Big\rvert X_{\gamma}$ where the $X_{\gamma}$ are independent random variables each of which is uniformly distributed on the unit circle. Applying~\cite[Lem.~6.2]{MartinNg} with the assumption 
	$\sum_{\gamma\in \mathcal{Z}_{\mathcal{S},\underline{c}}^{\mathrm{LI}}} \Big\lvert\frac{ 2 \ord_{\mathcal{S},\underline{c}}(\gamma)}{\frac{1}{2} + i \gamma}\Big\rvert = \infty $, we conclude that $\mathrm{supp}(\mu^{\mathrm{LI}}) = \mathbf{R}$ and Proposition~\ref{Prop inclusive} follows.
\end{proof}

\begin{proof}[Proof of Proposition~\ref{Prop oscillation without GRH}]
	The idea of the proof is similar to the proof of \cite[Th.~15.2]{MV-book} and is based on a theorem of Landau (precisely the contrapositive of \cite[Lem.~15.1]{MV-book}, which is also given in \cite{KP}).
	Fix $\epsilon >0$, we consider the real functions
	$$f_{\pm} : x \mapsto \sum_{n\leq x}\sum_{m\geq 0}c_m \frac{\Lambda_{f_m}(n)}{\log n} + \sum_{m\geq 0}c_{m}\ord_{s=1}(L(f_m,s)) \Li(x)  \pm x^{\Theta - \epsilon},$$
	where $\Theta$ is defined in the statement of Proposition~\ref{Prop oscillation without GRH} and for each $m\geq 0$, $\Lambda_{f_m}$ is the von Mangoldt function associated to $f_m$.
	Precisely, one has 
	$$ \Lambda_{f_m} = \begin{cases}
		\sum_{j=1}^{d_m} \alpha_{j,m}(p)^k \log p &\text{ if } n = p^k \\
		0 &\text{ if } n \text{ is not a prime power,}
	\end{cases}$$
where the $\alpha_{j,m}$ are the local roots of $L(f_m,\cdot)$.
In particular, using the Ramanujan--Petersson Conjecture, the Prime Number Theorem and the fact that the series $\sum_{m\geq 0}\lvert c_m\rvert d_m$ converges, we see that the functions $f_{\pm}$ are well-defined and we have that 
\begin{align}\label{Eq comp f and E}
f_{\pm}(x) = \frac{\sqrt{x}}{\log x}E_{\mathcal{S},\underline{c}}(x) \pm  x^{\Theta - \epsilon} + O_{\mathcal{S},\underline{c}}(x^{\frac12}) = O_{\mathcal{S},\underline{c},\epsilon'}(x^{\Theta + \epsilon'}),	
\end{align}
with $\epsilon'>0$ arbitrarily small.
For $\re(s) > \Theta$, write
\begin{align*}
	F_{\pm}(s) &= \int_{1}^{\infty} f_{\pm}(x) x^{-s-1} \diff s \\
	&= \frac{1}{s}\sum_{m\geq 0} c_m \Big(\log(L(f_m,s)) - \ord_{s=1}(L(f_m,s))\log(s-1) \Big)  + \frac{r_{\mathcal{S},\underline{c}}(s)}{s} \pm \frac{1}{s-\Theta+\epsilon}
\end{align*}
where we used absolute convergence to exchange the order of summation between the integral and the sum over $m\geq0$, and where the function~$r_{\mathcal{S},\underline{c}}$ is entire.
The second expression gives an analytic continuation of~$F_{\pm}$ to a larger set avoiding lines at the left of points~$\beta + i \gamma$ with $\ord_{\mathcal{S},\underline{c}}(\beta + i\gamma) \neq 0$ where the functions have logarithmic singularities.
In particular, by hypothesis, the functions~$F_{\pm}$ are regular at~$s= \Theta$, but are not regular in any half-plane~$\re(s)>\Theta - \epsilon'$ with $\epsilon'>0$.
Landau's Theorem \cite[Th.~(Landau)]{KP} then implies that the functions~$f_{\pm}$ have infinitely many sign changes.
We deduce that there are infinitely many $x >0$ such that $f_{-}(x) >0$, using~\eqref{Eq comp f and E}, we obtain
\begin{align*}
	\sum_{p\leq x}\sum_{m\geq 0}c_{m}\lambda_{f_m}(p) + \sum_{m\geq 0}c_{m}\ord_{s=1}(L(f_m,s)) \Li(x)  = \Omega_{+}(x^{\Theta -\epsilon}),
\end{align*}
and similarly~$f_{+}$ takes negative values infinitely many times, which then yields the~$\Omega_{-}$-result and concludes the proof. 
\end{proof}

\appendix
\section{Vanishing at the central point with PARI/GP}\label{App Pari}
We include here the code alluded to in Section~\ref{subsection hypothese} with some explanation for beginning \texttt{PARI/GP} users who might be interested in Hecke characters.
This appendix is the fruit of very helpful discussions with Emmanuel Royer in the CNRS International Research Laboratory in Montreal and is mostly an adaptation to our context of \cite{PageHeckeChar}, with help from Bill Allombert, Aurel Page and the User's Guide to PARI/GP \cite{PARI2}.

The base field over which our Hecke characters are defined is $\mathbf{Q}[i]$ which is \verb*|bnfinit(X^2+1)|. The Hecke characters are then elements in the group of characters \verb*|gcharinit| with modulus their conductor (or a multiple of their conductor), we can choose $(2+2i)$ i.e. \verb*|2+2*X| for $\xi_m$.
Recall that our Hecke character~$\xi$ is defined by
$$
	\xi((\alpha)) = \begin{cases}
		\frac{\alpha}{\lvert \alpha \rvert} &\text{ if } \alpha \equiv 1 \bmod \mathfrak{m} \\
		0 &\text{ if } (\alpha, \mathfrak{m}) \neq 1.
	\end{cases}
$$
It has frequency $1$, conductor $(2+2i)$ and for $m\geq 1$, the character $\xi_m$ is the primitive character associated to $\xi^m$.
The function \verb*|gcharidentify| allows us to determine uniquely $\xi$ by indicating that its infinite component (at \verb*|[1]|) is $z \mapsto \big(\frac{z}{\lvert z\rvert}\big)^{-1}$ (we have to take an inverse when going from our classical definition above to the adelic definition used by PARI/GP), this is \verb*|[-1,0]|.
As the character group is encoded additively, taking the power $m$ is multipication by \verb*|m|.
We  then create the $L$-functions associated to Hecke characters using \verb*|lfuncreate| which yields a vector containing interesting data of the $L$-function. In particular the $4$-th coordinate is a number $k$ such that the functional equation of the $L$-function relates $s \leftrightarrow k-s$, so $\frac{k}2$ is the central point. The function \verb*|lfun| evaluates our $L$-function at a point, a third parameter allow to evaluate the derivatives instead. 

Given this information, we can now be convinced that the code below will return the integers~$m$ that do not satisfy the following assumption :
\begin{equation*}
	\ord_{s= \frac12}(L(s,\xi^m)) = \begin{cases}
		0 &\text{ if } m \equiv 0, 1, 2, 3, 4, 6 \bmod 8, \\
		1 &\text{ if } m \equiv 5 , 7 \bmod 8.
	\end{cases}
\end{equation*}

\begin{verbatim}
ConjectureXimax(min,Max,valeurs=0,epsilon=10^(-5))={
	my(
		Corps       = bnfinit(X^2+1),
		GroupeCar   = gcharinit(Corps,2+2*X),
		xi          = gcharidentify(GroupeCar,[1],[[-1,0]]),
		Lval        = List()
	);
	for(m=min,Max,
		my(
			xim         = m*xi,
			foncL       = lfuncreate([GroupeCar,xim]),
			centre      = foncL[4]/2,
			valL        = lfun(foncL,centre),
			z           = 0
		);
		if(abs(valL)<epsilon,if(m%8==5 || m%8==7,z=1;valL=lfun(foncL,centre,1);
			if(abs(valL)<epsilon,printf("*****");print(m)),
				printf("*****");print(m)));
		if(valeurs,listput(Lval,[z,real(valL)]))
	);
	if(valeurs,return(Lval),return());
};
\end{verbatim}
This code ran for about one hour in a personal computer to check values of $m$ in $[1,2000]$ and only returned the value $m=1897 \equiv 1 \bmod 8$ for which it gave $L(\frac12,\xi_m) = 1.2362\ldots \times 10^{-6}$ smaller than our default $\epsilon = 10^{-5}$ but still not $0$.

Now, we write a similar code for the characters $\psi_m$'s.
They are characters of conductor $(4)$.
The character~$\psi_0$ is defined by 
$$
	\psi_0((\alpha)) = \begin{cases}
		1 &\text{ if } \alpha \equiv 1 \bmod (4) \\
		- 1 &\text{ if } \alpha \equiv 3+2i \bmod (4) \\
		0 &\text{ if } (\alpha, (4)) \neq 1,
	\end{cases}
$$
in particular it is not trivial but has frequency $0$ (so the infinite component is trivial \verb*|[0,0]|). 
To use \verb*|gcharidentify|, we also note that $\psi_0(\mathfrak{p})=-1$ for $\mathfrak{p}\mid 5$. 
For $m\geq 1$, observe that $\psi_{m} = \psi_0 \xi^m$, we induce $\xi$ from the group of characters modulo $(2+2i)$ by again using \verb*|gcharidentify| with the same infinte component and value at a prime dividing $5$.
	This explains the code below to check the assumption
$$
	\ord_{s= \frac12}(L(s,\psi_m)) = \begin{cases}
		0 &\text{ if } m \equiv 0, 1, 2 \bmod 4, \\
		1 &\text{ if } m \equiv 3 \bmod 4. 
	\end{cases} 
$$

\begin{verbatim}
ConjecturePsimax(min,Max,avancement=0,valeurs=0,epsilon=10^(-5))={
	my(
		Corps       = bnfinit(X^2+1),
		GroupeCar2  = gcharinit(Corps,2+2*X),
		GroupeCar4  = gcharinit(Corps,4),
		pr5		= idealprimedec(Corps,5)[1],
		psi0	= gcharidentify(GroupeCar4,[1,pr5],[[0,0],1/2]),
		Xi2		= gcharidentify(GroupeCar2,[1],[[-1,0]]),
		Xi4		= gcharidentify(GroupeCar4,[1,pr5],[[-1,0],gchareval(GroupeCar2,Xi2,pr5,0)]),
		Lval	= List()
	);
	for(m=min,Max,
		my(
			psim	= psi0 +m*Xi4,
			foncL	= lfuncreate([GroupeCar4,psim]),
			centre	= foncL[4]/2,
			valL	= lfun(foncL,centre),
			z		= 0
		);
		if(abs(valL)<epsilon,if(m%4==3,z=1;valL=lfun(foncL,centre,1);
			if(abs(valL)<epsilon,printf("*****");print(m)),
				printf("*****");print(m)));
		if(valeurs,listput(Lval,[z,real(valL)]))
	);
	if(valeurs,return(Lval),return());
	};
\end{verbatim}
This code ran for about one hour and a half in a personal computer to check values of $m$ in $[1,2000]$ and did not return anything.

\bibliographystyle{amsalpha} 
\bibliography{biblio}
	
\end{document}